\documentclass[11pt]{amsart}
\usepackage[all]{xy}
\usepackage {fancybox}
\usepackage{ascmac}
\usepackage{amssymb} 
\usepackage{tikz-qtree}
\usepackage{forest}
\usepackage{mathrsfs}
\usepackage[T1]{fontenc}
\usepackage{amscd,verbatim}
\usepackage{floatflt}
\usepackage{xcolor}
\usepackage[hidelinks]{hyperref}
\usepackage{tikz-cd}
\newtheorem{lemma}{Lemma}[section]

\newtheorem{thm}[lemma]{Theorem}

\newtheorem{theorem}[lemma]{Theorem}

\newtheorem{prop}[lemma]{Proposition}
\newtheorem{proposition}[lemma]{Proposition}

\newtheorem{defn}[lemma]{Definition}

\newtheorem{definition}[lemma]{Definition}

\newtheorem{remark}[lemma]{Remark}

\numberwithin{equation}{section}

\renewcommand{\L}{\mathbb{L}}

\newcommand{\Q}{\mathbb{Q}}
\newcommand{\Z}{\mathbb{Z}}

\newcommand{\sC}{\mathcal{C}}
\newcommand{\sD}{\mathcal{D}}

\newcommand{\sO}{\mathcal{O}}

\newcommand{\sQ}{\mathcal{Q}}
\newcommand{\sW}{\mathcal{W}}

\newcommand{\sT}{\mathcal{T}}

\newcommand{\sM}{\mathcal{M}}

\newcommand{\ModFI}{\operatorname{ModFI}}
\newcommand{\BKMod}{\operatorname{Mod}_{/\mS}}
\newcommand{\BKModone}{\operatorname{Mod}_{/\mS_1}}
\newcommand{\BKModinf}{\operatorname{Mod}_{/\mS_{\infty}}}
\newcommand{\ev}{\operatorname{ev}}
\newcommand{\free}{\operatorname{free}}
\newcommand{\an}{\operatorname{an}}
\newcommand{\val}{\operatorname{val}}
\newcommand{\tf}{\operatorname{tf}}
\newcommand{\len}{\operatorname{length}} 
\newcommand{\tors}{\operatorname{tor}}

\newcommand{\sX}{\mathcal{X}}

\newcommand{\mf}{\mathfrak}

\newcommand{\fM}{\mf{M}}

\newcommand*{\xdashrightarrow}[2][]{%
  \mathrel{%
    \mathpalette{\da@xarrow{#1}{#2}{}\da@rightarrow{\,}{}}{}%
  }%
}

\newcommand{\Mod}{\operatorname{Mod}}

\newcommand{\fib}{\operatorname{fib}}
\newcommand{\cofib}{\operatorname{cofib}}

\newcommand{\Ext}{\operatorname{Ext}}

\newcommand{\dR}{\operatorname{{dR}}}
\newcommand{\Sing}{\operatorname{{Sing}}}
\newcommand{\qdR}{q\text{-}\operatorname{{dR}}}

\newcommand{\Tor}{\operatorname{Tor}}

\newcommand{\Ker}{\operatorname{Ker}}

\newcommand{\im}{\operatorname{Im}}
\renewcommand{\Im}{\operatorname{Im}}
\newcommand{\Coker}{\operatorname{Coker}}

\newcommand{\Spec}{\operatorname{Spec}}
\newcommand{\Spf}{\operatorname{Spf}}

\newcommand{\syn}{{\operatorname{syn}}}
\newcommand{\et}{{\operatorname{\acute{e}t}}}

\newcommand{\ch}{{\operatorname{ch}}}

\newcommand{\topp}{{\operatorname{top}}}
\newcommand{\gr}{{\operatorname{gr}}}
\newcommand{\fil}{{\operatorname{fil}}}
\newcommand{\BMS}{\mathrm{BMS}}

\newcommand{\can}{\operatorname{can}}
\newcommand{\bC}{\mathbb{C}}

\newcommand{\perf}{\operatorname{perf}}

\newcommand{\mS}{\mf{S}}

\newcommand{\Sp}{\operatorname{Sp}}

\newcommand{\cry}{\operatorname{cry}}
\numberwithin{equation}{section}

\renewcommand{\lim}{\operatornamewithlimits{\varprojlim}}
\newcommand{\colim}{\operatornamewithlimits{\varinjlim}}
\newcommand{\hocolim}{\operatorname{hocolim}}
\newcommand{\ol}{\overline}

\newcommand{\Ainf}{A_{\operatorname{inf}}}

\newcommand{\THH}{\operatorname{THH}}
\newcommand{\HH}{\operatorname{HH}}
\newcommand{\TP}{\operatorname{TP}}
\newcommand{\TC}{\operatorname{TC}}
\newcommand{\TCn}{\operatorname{TC}^{-}}
\newcommand{\HP}{\operatorname{HP}}
\newcommand{\ku}{\operatorname{ku}}
\newcommand{\Ktop}{K_{\operatorname{top}}}

\newcommand{\bS}{\mathbb{S}}
\newcommand{\LK}{{L_{K(1)}}}

\begin{document}

\title[Degeneration Theorems in Mixed Characteristic]{Degeneration Theorems of Connes and Feigin--Tsygan Type in Mixed Characteristic, with \texorpdfstring{$q$}{q}-Analogues}

\author{Keiho Matsumoto}
\email{lightkun0526@gmail.com}
\address{Graduate School of Science, Osaka University, Toyonaka, Osaka 560-0043, Japan}

\begin{abstract}
We prove mixed-characteristic analogues of the Connes and Feigin--Tsygan degeneration theorem. Let $W=W(k)$ be the Witt vectors of a perfect field of characteristic $p>0$. For a smooth proper variety $X$ over $W$, the de Rham-to-$\HP$ spectral sequence is split degenerate under the small-dimension hypothesis $\dim(X/W)<p-1$. More generally, if $X$ is smooth and proper over the ring of integers $\sO_K$ of a finite extension of $\mathrm{Frac}(W)$ with ramification index $e$, we prove the corresponding split degeneration under $2e\dim(X/\sO_K)<p-1$. Under the same ramification hypothesis, we also prove split degeneration of the $\Ainf$-to-$\TP$ spectral sequence. Finally, after inverting an explicit factorial, we obtain a topological $q$-de Rham analogue.
\end{abstract}

\maketitle

\tableofcontents

\section{Introduction}

For a smooth algebra $A$ over a commutative $\Q$-algebra $k$, the classical theorem of Connes, Feigin--Tsygan, and Goodwillie identifies periodic cyclic homology $\HP(A/k)$ with the $2$-periodization of the de Rham complex \cite{C83,FT83,G85}. More recently, Bhatt--Morrow--Scholze, in the $p$-complete quasisyntomic setting, and Antieau, for arbitrary commutative base rings, constructed an analogue of this picture in the form of a complete $\Z$-indexed multiplicative filtration on periodic cyclic homology \cite[Theorem~1.17]{BMS19}, \cite[Theorem~1.1]{A19}. For a smooth $k$-scheme $X$, the $n$-th graded piece of this filtration is given by
\[
R\Gamma(X,\Omega^\bullet_{X/k})[2n],
\]
or more precisely, by the Hodge-completed derived de Rham complex, which agrees with the usual de Rham complex in the smooth case. In general, however, the spectral sequence associated with this filtration need not be split degenerate.

In characteristic $p$, the situation improves under suitable dimension and liftability hypotheses. Let $k$ be a perfect field of characteristic $p>0$, and let $X$ be a smooth proper variety over $k$ admitting a lift to $W_2(k)$. If $\dim X<p$, then the Hochschild--Kostant--Rosenberg filtration on $\HH(X/k)$ degenerates. Indeed, this follows from the HKR theorem in characteristic $p$, due to Yekutieli and strengthened by Antieau--Vezzosi, together with the Deligne--Illusie theorem on the degeneration of the Hodge-to-de Rham spectral sequence \cite{Y02,AV20,DI87}. On the other hand, Mathew proved that if $T$ is a smooth proper dg category over $k$ admitting a lift to $W_2(k)$ and satisfying $\HH_i(T/k)=0$ for $i\notin[-p,p]$, then the Tate spectral sequence
\[
\HH_*(T/k)[u^{\pm1}] \Longrightarrow \HP_*(T/k)
\]
degenerates \cite[Theorem~1.2]{M20b}. Applying this to $T=\perf(X)$, the degeneration of the HKR filtration yields the required Hochschild amplitude bound, and a comparison of dimensions then shows that the de Rham-to-$\HP$ spectral sequence of Antieau also degenerates for $X$. The purpose of this paper is to establish mixed-characteristic analogues of this picture.

Let $W=W(k)$ be the ring of Witt vectors of a perfect field $k$ of characteristic $p>0$, let $K/\mathrm{Frac}(W)$ be a finite extension with ramification index $e$, and let $\sO_K$ be its ring of integers. 
\begin{theorem}
\begin{itemize}
  \item[(1)] Let $X$ be a smooth proper variety over $\sO_K$. Assume that $\dim X<p-1$ if $\sO_K=W$, and assume that $e(2\dim X-1)<p-1$ if $e\geq 2$. Then the crystalline-to-$\TP$ spectral sequence is split degenerate. In particular, there is an isomorphism of $W$-modules
\[
\pi_j\TP(X_k;\Z_p) \simeq \bigoplus_{i\in\Z} H_{\cry}^{-j+2i}(X_k/W).
\]
  \item[(2)] Let $X$ be a smooth proper variety over $\sO_K$. Assume $2e\dim(X/\sO_K)<p-1$. Then the de Rham-to-$\HP$ spectral sequence is split degenerate. In particular, there is an isomorphism of $\sO_K$-modules
  \begin{align*}
    \HP_j(X/\sO_K;\Z_p)&\simeq \bigoplus_{i\in \Z} H_{\dR}^{-j+2i}(X/\sO_K).
  \end{align*}
\end{itemize}
\end{theorem}

\begin{remark}
The two smallness assumptions in the preceding theorem come from different
arguments.  The crystalline-to-$\TP$ statement is proved directly from the
crystalline comparison and the topological $K$-theoretic splitting criterion; it
does not pass through the Breuil--Kisin splitting argument used for the
de Rham-to-$\HP$ and $\Ainf$-to-$\TP$ statements.  This is why the first
statement admits the slightly sharper bound
$e(2\dim(X/\sO_K)-1)<p-1$ in the ramified case, whereas the Breuil--Kisin
arguments below use the $E$-height range $[0,2\dim(X/\sO_K)]$ and hence
require $2e\dim(X/\sO_K)<p-1$.
\end{remark}

The preceding theorem concerns the mixed-characteristic analogue of the de Rham-to-$\HP$ filtration. We establish similar split-degeneration results for further filtrations on topological additive invariants. In the $p$-adic setting, Bhatt--Morrow--Scholze constructed complete multiplicative filtrations on $\THH$, $\TC^-$, and $\TP$ for quasisyntomic rings, and showed that, for smooth $p$-adic formal schemes, the resulting graded pieces are governed by $\Ainf$-cohomology and its Breuil--Kisin twists \cite[Theorem 1.12]{BMS19}. Thus $\TP$ carries a mixed-characteristic filtration which may be viewed as a topological analogue of the motivic filtrations discussed above. Our next result shows that, under the same small ramification hypothesis, this filtration is again split degenerate. We also prove a $q$-de Rham analogue for the $S^1$-equivariant even filtration on $\TP(\mathcal X\otimes \ku/\ku)$, which is constructed by Wagner \cite{W25arxiv}.

\begin{theorem}
\begin{itemize}
  \item[(1)] Let $X$ be a smooth proper variety over $\sO_K$. Assume $2e\dim(X/\sO_K)<p-1$. Then the $\Ainf$-to-$\TP$ spectral sequence is split degenerate:
  \[
    \TP_j(X;\Z_p)\simeq \bigoplus_{i\in \Z} H_{\Ainf}^{-j+2i}(X).
  \]

  \item[(2)] Fix an even natural number $N$. Let $X$ be a smooth proper variety over $\Z[1/N]$, and write $d=\dim(X/\Z)$. Assume that there exists a derived scheme $\mathcal X$ over $\bS[1/N]$ together with an equivalence
  \[
    \mathcal X\otimes_{\bS}\Z \simeq X.
  \]
  Then, after inverting $(2d+1)!$, the $S^1$-equivariant even filtration on $\TP(\mathcal X\otimes \ku/\ku)$ is split degenerate:
  \[
    \pi_j\TP(\mathcal X\otimes \ku/\ku)\Bigl[\frac{1}{(2d+1)!}\Bigr]
    \simeq
    \bigoplus_{i\in \Z} H_{\qdR}^{-j+2i}(X/\Z)\Bigl[\frac{1}{(2d+1)!}\Bigr].
  \]
\end{itemize}
\end{theorem}
The paper is organized as follows. In Section~2, we prove an integral form of the Chern character decomposition for reduced complex topological $K$-theory.  More precisely, for a finite CW complex of dimension $d$, we show that the Chern character becomes an isomorphism after inverting $\lfloor (d+1)/2\rfloor!$.  This result will be used later to control the integrality range in the topological $q$-de Rham analogue.

In Section~3, we study the crystalline-to-$\TP$ spectral sequence.  We first recall the pieces of noncommutative $p$-adic Hodge theory needed in the argument. After that, we prove split
degeneration of the crystalline-to-$\TP$ spectral sequence.

Section~4 is devoted to Breuil--Kisin modules.  We recall the relevant categories of Breuil--Kisin modules and the low-ramification structure theorem for torsion Breuil--Kisin modules.  We then prove a variant of this structure theorem which also allows finite free Breuil--Kisin modules.  This variant is the key algebraic input needed to control the filtrations appearing in the proof of the main theorems.

In Section~5, we prove the main results. We first prove split degeneration of the Breuil--Kisin-to-$\TCn$ spectral sequence.  By specializing this result along the maps from the Breuil--Kisin prism to $\sO_K$ and to $\Ainf$, we obtain the split degeneration of the de Rham-to-$\HP$ spectral sequence and of the $\Ainf$-to-$\TP$ spectral sequence.  Finally, we prove the topological $q$-de Rham analogue.

\bigskip
\noindent{\bfseries Acknowledgements.}

The author would like to express his sincere gratitude to Hiroyasu Miyazaki, whose question of whether commutative $p$-adic Hodge theory can be recovered from noncommutative $p$-adic Hodge theory provided the initial motivation for this work. The author is also grateful to Hui Gao for helpful conversations; in particular, the discussion of the split degeneration of the $\Ainf$-to-$\TP$ spectral sequence was partly inspired by Gao's perspective on integral filtered Sen theory. This project was initiated during the author's stay at the Tianyuan Mathematical Research Center. The author would like to thank the Center for its generous hospitality, excellent working conditions, and inspiring environment, all of which greatly contributed to the development of this work. This work was supported by Kakenhi Grants 22K13898 and 25KJ0210.

\section{\texorpdfstring{Chern character decomposition of $\Ktop$-theory over $\mathbf{Z}[1/M!]$}{Chern character decomposition of K-theory over Z[1/M!]}}

In this section we work \emph{entirely in reduced theories}. For a (finite) CW complex $Y$ we write
$\widetilde{K}_{\topp}^0(Y)$ and $\widetilde{K}_{\topp}^{-1}(Y)$ for reduced complex topological $K$-theory, and
$\widetilde H^*(Y;R)$ for reduced singular cohomology with coefficients in a ring $R$.
We also set
\[
\widetilde H^{\mathrm{even}}(Y;R):=\bigoplus_{n\ge 0}\widetilde H^{2n}(Y;R),
\qquad
\widetilde H^{\mathrm{odd}}(Y;R):=\bigoplus_{n\ge 0}\widetilde H^{2n+1}(Y;R).
\]
For $r\in\mathbf{Z}$, we also use the parity notation
\[
\widetilde H^{[r]}(Y;R):=
\begin{cases}
\widetilde H^{\mathrm{even}}(Y;R),& r\equiv 0\pmod 2,\\
\widetilde H^{\mathrm{odd}}(Y;R),& r\equiv 1\pmod 2.
\end{cases}
\]
Let $X$ be a finite CW complex of (cellular) dimension $d=\dim X$, and set
\[
M:=\Bigl\lfloor\frac{d+1}{2}\Bigr\rfloor,
\qquad
R:=\mathbf{Z}\!\left[\frac{1}{M!}\right].
\]

Over $\mathbf{Q}$, Atiyah showed that complex topological $K$-theory $\Ktop$ splits via the Chern character, i.e.\ it decomposes in terms of singular cohomology; see \cite{A62}.
In this paper we refine this by localizing at a finite set of primes: more precisely, we work over
\[
R=\mathbf{Z}\bigl[1/M!\bigr]
\]
and show that the same decomposition holds after tensoring $\widetilde{K}_{\topp}^m(X)$ with $R$.

\begin{remark}
The bound $M!$ is the elementary denominator bound obtained from the splitting
principle and the usual Chern character.  We do not claim that it is optimal:
sharper $p$-local decompositions can be obtained by more refined methods, for
instance using Adams operations.  The present form is chosen because it is
explicit and is sufficient for the topological $q$-de Rham application in
Section~5.
\end{remark}

\begin{theorem}\label{ktopchdec}
For every finite CW complex $X$ of dimension $d$ as above, the Chern character induces natural isomorphisms
\[
\begin{aligned}
\ch\colon \widetilde{K}_{\topp}^{0}(X)\otimes_{\mathbf{Z}} R &\xrightarrow{\ \sim\ } \widetilde H^{\mathrm{even}}(X;R),\\
\ch\colon \widetilde{K}_{\topp}^{-1}(X)\otimes_{\mathbf{Z}} R &\xrightarrow{\ \sim\ } \widetilde H^{\mathrm{odd}}(X;R).
\end{aligned}
\]
\end{theorem}


\begin{lemma}\label{chint}
Let $X$ be a finite CW complex of dimension $d$ and set $M=\lfloor(d+1)/2\rfloor$, $R=\mathbf{Z}[1/M!]$.
Then the usual rational Chern character on reduced $K$-theory factors through $R$:
\[
\begin{aligned}
\ch:\ \widetilde{K}_{\topp}^{0}(X)&\longrightarrow \widetilde H^{\mathrm{even}}(X;\mathbf{Q}) \\
&\text{factors as} \\
\ch:\ \widetilde{K}_{\topp}^{0}(X)&\longrightarrow \widetilde H^{\mathrm{even}}(X;R).
\end{aligned}
\]
and similarly
\[
\ch:\ \widetilde{K}_{\topp}^{-1}(X)\longrightarrow \widetilde H^{\mathrm{odd}}(X;\mathbf{Q})
\quad\text{factors as}\quad
\ch:\ \widetilde{K}_{\topp}^{-1}(X)\longrightarrow \widetilde H^{\mathrm{odd}}(X;R).
\]
The same statement holds for any finite CW complex $Y$ of dimension $\le d$ (in particular for any cofiber $X/A$
of a CW pair $(X,A)$ with $\dim X\le d$).
\end{lemma}

\begin{proof}
We first treat $\widetilde{K}_{\topp}^0$.
Recall that for a complex vector bundle $E$ on $X$ the Chern character is defined by
\[
\ch(E)=\sum_{n\ge 0}\ch_n(E),\qquad \ch_n(E)\in H^{2n}(X;\mathbf{Q}),
\]
and $\ch_n(E)$ is a universal polynomial in $c_1(E),\dots,c_n(E)$ with coefficients in $\mathbf{Z}[1/n!]$
(equivalently, the denominators of $\ch_n$ divide $n!$).

To see this in a way suited for denominator control, use the splitting principle as follows.
There exists a map $p\colon X'\to X$ such that:
(i) $p^*\colon H^*(X;\mathbf{Z})\to H^*(X';\mathbf{Z})$ is injective, and
(ii) $p^*E$ splits as a direct sum of line bundles $L_1\oplus\cdots\oplus L_r$.
Write $x_i:=c_1(L_i)\in H^2(X';\mathbf{Z})$ for the Chern roots. Then by definition of $\ch$ for line bundles,
\[
\ch(L_i)=e^{x_i}=\sum_{n\ge 0}\frac{x_i^n}{n!},
\]
and hence
\[
\ch(p^*E)=\sum_{i=1}^r e^{x_i}
=\sum_{n\ge 0}\frac{1}{n!}\Bigl(\sum_{i=1}^r x_i^n\Bigr).
\]
Therefore for each $n\ge 0$ we have
\[
\ch_n(p^*E)\in H^{2n}\!\left(X';\mathbf{Z}\!\left[\frac{1}{n!}\right]\right).
\]
Since $p^*$ is injective on cohomology, it follows that already on $X$,
\[
\ch_n(E)\in H^{2n}\!\left(X;\mathbf{Z}\!\left[\frac{1}{n!}\right]\right).
\]

Now $\dim X=d$ implies $H^{k}(X;\mathbf{Z})=0$ for $k>d$, hence $\ch_n(E)=0$ for $2n>d$.
Thus only the components with $n\le \lfloor d/2\rfloor$ can occur, and all denominators are killed by $\bigl\lfloor d/2\bigr\rfloor!\ \mid\ M!$. Consequently $\ch(E)$ actually lands in $H^{\mathrm{even}}(X;R)$, and passing to reduced groups
(i.e.\ taking kernels of the maps to the point) gives
\[
\ch:\ \widetilde{K}_{\topp}^0(X)\longrightarrow \widetilde H^{\mathrm{even}}(X;R).
\]

For $\widetilde{K}_{\topp}^{-1}$ we use the standard identification $\widetilde{K}_{\topp}^{-1}(X)\cong \widetilde{K}_{\topp}^0(\Sigma X)$.
Since $\Sigma X$ is a finite CW complex of dimension $d+1$, the above argument applied to $\Sigma X$ shows
that the denominators needed for $\ch$ on $\widetilde{K}_{\topp}^0(\Sigma X)$ divide $
\Bigl\lfloor\frac{d+1}{2}\Bigr\rfloor! \ =\ M!$, hence $\ch$ lands in $\widetilde H^{\mathrm{even}}(\Sigma X;R)\cong \widetilde H^{\mathrm{odd}}(X;R)$.
This yields the desired factorization for $\widetilde{K}_{\topp}^{-1}(X)$.

Finally, if $Y$ is any finite CW complex with $\dim Y\le d$, the same argument applies verbatim (with $d$ replaced by
$\dim Y$). In particular, for a CW pair $(X,A)$ with $\dim X\le d$, its cofiber $X/A$ has $\dim(X/A)\le d$,
so the same denominator control applies to $Y=X/A$.
\end{proof}


Let $(X,A)$ be a CW pair with inclusion $i\colon A\hookrightarrow X$. The quotient $X/A$ is the cofiber of $i$,
and we have the canonical cofiber (Puppe) sequence
\[
A \xrightarrow{\,i\,} X \longrightarrow X/A \xrightarrow{\,q\,} \Sigma A.
\]
For any reduced cohomology theory $\widetilde E^*$, the long exact sequence associated to this cofibration
has connecting morphisms
\[
\delta_E:\ \widetilde E^{m}(A)\longrightarrow \widetilde E^{m+1}(X/A),
\]
and the connecting morphism admits the standard geometric description
\[
\delta_E \ =\ q^{*}\circ \sigma_E,
\qquad
\sigma_E:\ \widetilde E^{m}(A)\xrightarrow{\ \cong\ }\widetilde E^{m+1}(\Sigma A)
\]
where $\sigma_E$ is the suspension isomorphism. This description is precisely the one used in \cite{A62}.

\begin{theorem}\label{chbdy}
Let $(X,A)$ be a CW pair with $\dim X\le d$, and set
\[
M=\Bigl\lfloor\frac{d+1}{2}\Bigr\rfloor,\qquad R=\mathbf{Z}\!\left[\frac{1}{M!}\right].
\]
Then the Chern character
\[
\ch:\ \widetilde{K}_{\topp}^{0}(-)\otimes_{\mathbf{Z}} R \to \widetilde H^{\mathrm{even}}(-;R),
\qquad
\ch:\ \widetilde{K}_{\topp}^{-1}(-)\otimes_{\mathbf{Z}} R \to \widetilde H^{\mathrm{odd}}(-;R)
\]
is compatible with the connecting morphisms in the long exact sequences of the cofibration
$A\to X\to X/A$.
Equivalently, for $m=0,-1$ the square
\[
\begin{tikzcd}[column sep=large]
\widetilde{K}_{\topp}^{m}(A)\otimes R \ar[r,"\delta_K"] \ar[d,"\ch"'] &
\widetilde{K}_{\topp}^{m+1}(X/A)\otimes R \ar[d,"\ch"] \\
\widetilde H^{[m]}(A;R) \ar[r,"\delta_H"] &
\widetilde H^{[m+1]}(X/A;R)
\end{tikzcd}
\]
commutes.
\end{theorem}

\begin{proof}
Let $q\colon X/A\to \Sigma A$ be the Puppe map. By the geometric description of the connecting morphism recalled above,
we have
\[
\delta_K=q^*\circ \sigma_K,\qquad \delta_H=q^*\circ \sigma_H,
\]
where $\sigma_K,\sigma_H$ are the suspension isomorphisms in reduced $K$-theory and reduced cohomology.

The Chern character is a natural transformation of reduced cohomology theories (see \cite[\S2--\S3]{AH61}),
hence it commutes with pullback along $q$ and also with suspension. Therefore the diagrams
\[
\begin{tikzcd}[column sep=large]
\widetilde{K}_{\topp}^{m+1}(\Sigma A)\otimes R \ar[r,"q^*"] \ar[d,"\ch"'] &
\widetilde{K}_{\topp}^{m+1}(X/A)\otimes R \ar[d,"\ch"] \\
\widetilde H^{[m+1]}(\Sigma A;R) \ar[r,"q^*"] &
\widetilde H^{[m+1]}(X/A;R)
\end{tikzcd}
\]
\[
\begin{tikzcd}[column sep=large]
\widetilde{K}_{\topp}^{m}(A)\otimes R \ar[r,"\sigma_K"] \ar[d,"\ch"'] &
\widetilde{K}_{\topp}^{m+1}(\Sigma A)\otimes R \ar[d,"\ch"] \\
\widetilde H^{[m]}(A;R) \ar[r,"\sigma_H"] &
\widetilde H^{[m+1]}(\Sigma A;R)
\end{tikzcd}
\]
commute. Composing and using $\delta_K=q^*\circ\sigma_K$, $\delta_H=q^*\circ\sigma_H$ yields the claim.

Finally, Lemma~\ref{chint} ensures that when $\dim X\le d$, the Chern character takes values
in $R$ on $A$ and on $X/A$ (since $\dim(X/A)\le d$), so the above argument indeed applies over $R=\mathbf{Z}[1/M!]$.
\end{proof}


Let $X^{(p)}$ denote the $p$-skeleton of $X$, and set $A:=X^{(d-1)}$.
Since $X$ is obtained from $A$ by attaching $d$-cells, collapsing $A$ to a point identifies the cofiber $X/A$
(up to pointed homotopy) with a wedge of $d$-spheres:
\[
X/A\ \simeq\ \bigvee_{\alpha\in I_d} S^d,
\]
with one summand for each $d$-cell of $X$. We will use this description to compute reduced $K$-theory and reduced
cohomology of $X/A$.

\begin{lemma}\label{cofibch}
Let $X$ be a finite CW complex of dimension $d$, set $A:=X^{(d-1)}$, and let
$M=\lfloor(d+1)/2\rfloor$, $R=\mathbf{Z}[1/M!]$.
Then the Chern character induces isomorphisms
\[
\begin{aligned}
\ch\colon \widetilde{K}_{\topp}^{0}(X/A)\otimes R &\xrightarrow{\ \sim\ } \widetilde H^{\mathrm{even}}(X/A;R),\\
\ch\colon \widetilde{K}_{\topp}^{-1}(X/A)\otimes R &\xrightarrow{\ \sim\ } \widetilde H^{\mathrm{odd}}(X/A;R).
\end{aligned}
\]
\end{lemma}

\begin{proof}
Using the wedge decomposition $X/A\simeq \bigvee_{\alpha} S^d$ and the wedge axiom for reduced theories, it suffices
to treat $S^d$.

If $d$ is even, Bott periodicity gives $\widetilde{K}_{\topp}^{0}(S^d)\cong \mathbf{Z}$ and $\widetilde{K}_{\topp}^{-1}(S^d)=0$,
and $\widetilde H^{\mathrm{even}}(S^d;R)\cong \widetilde H^{d}(S^d;R)\cong R$ while
$\widetilde H^{\mathrm{odd}}(S^d;R)=0$.
By the standard normalization of the Chern character (sending the Bott generator to the fundamental cohomology
class, up to sign), $\ch$ induces an isomorphism after tensoring with $R$.

If $d$ is odd, then
\[
\widetilde{K}_{\topp}^{-1}(S^d)\cong \mathbf{Z},
\qquad
\widetilde{K}_{\topp}^{0}(S^d)=0,
\]
and
\[
\widetilde H^{\mathrm{odd}}(S^d;R)\cong \widetilde H^{d}(S^d;R)\cong R,
\qquad
\widetilde H^{\mathrm{even}}(S^d;R)=0.
\]
Again, by normalization and Lemma~\ref{chint} (applied to $S^d$), the map becomes an
isomorphism after tensoring with $R$.

Therefore $\ch$ is an isomorphism on each sphere summand and hence on the wedge $X/A$.
\end{proof}


\begin{lemma}\label{fivech}
Let $X$ be a finite CW complex of dimension $d$, set $A:=X^{(d-1)}$, and let $R=\mathbf{Z}[1/M!]$ with
$M=\lfloor(d+1)/2\rfloor$.
Assume that the Chern character is an isomorphism on $A$:
\[
\ch:\ \widetilde{K}_{\topp}^{0}(A)\otimes R \xrightarrow{\ \sim\ } \widetilde H^{\mathrm{even}}(A;R),
\qquad
\ch:\ \widetilde{K}_{\topp}^{-1}(A)\otimes R \xrightarrow{\ \sim\ } \widetilde H^{\mathrm{odd}}(A;R).
\]
Then the same holds for $X$:
\[
\ch:\ \widetilde{K}_{\topp}^{0}(X)\otimes R \xrightarrow{\ \sim\ } \widetilde H^{\mathrm{even}}(X;R),
\qquad
\ch:\ \widetilde{K}_{\topp}^{-1}(X)\otimes R \xrightarrow{\ \sim\ } \widetilde H^{\mathrm{odd}}(X;R).
\]
\end{lemma}

\begin{proof}
Consider the cofibration $A\to X\to X/A$ and the associated long exact sequences in reduced $K$-theory and in reduced
cohomology with coefficients in $R$. Splicing the even/odd degrees, we obtain $2$-periodic exact sequences
\[
\begin{aligned}
\cdots\to {}&\widetilde{K}_{\topp}^{-1}(X/A)\to \widetilde{K}_{\topp}^{-1}(X)\to \widetilde{K}_{\topp}^{-1}(A)\xrightarrow{\delta_K}\\
&\widetilde{K}_{\topp}^{0}(X/A)\to \widetilde{K}_{\topp}^{0}(X)\to \widetilde{K}_{\topp}^{0}(A)\xrightarrow{\delta_K}\widetilde{K}_{\topp}^{-1}(X/A)\to\cdots
\end{aligned}
\]
and similarly
\begin{align*}
\cdots &\to \widetilde H^{\mathrm{odd}}(X/A;R)\to \widetilde H^{\mathrm{odd}}(X;R)\to \widetilde H^{\mathrm{odd}}(A;R)\xrightarrow{\delta_H}\\
&\to \widetilde H^{\mathrm{even}}(X/A;R)\to \widetilde H^{\mathrm{even}}(X;R)\to \widetilde H^{\mathrm{even}}(A;R)\xrightarrow{\delta_H}\\
&\to \widetilde H^{\mathrm{odd}}(X/A;R)\to\cdots .
\end{align*}
By Theorem~\ref{chbdy}, the Chern character defines a morphism between these exact sequences.

By assumption, $\ch$ is an isomorphism on the $A$-terms.
By Lemma~\ref{cofibch}, $\ch$ is an isomorphism on the $X/A$-terms.
Applying the five lemma to the corresponding segments of the above exact sequences shows that $\ch$ is an isomorphism
on the $X$-terms, both for $\widetilde{K}_{\topp}^{0}$ and for $\widetilde{K}_{\topp}^{-1}$.
\end{proof}

\begin{proof}[Proof of Theorem~\ref{ktopchdec}]
We proceed by induction on $d=\dim X$. If $d=0$, then $X$ is a finite discrete space. The Chern character identifies $\widetilde{K}_{\topp}^0(X)$ with $\widetilde H^0(X;\Z)$, while both odd groups vanish; after tensoring with $R$, the statement is immediate. Now assume that the theorem holds for all finite CW complexes of dimension at most $d-1$.
Let $X$ be a $d$-dimensional CW complex and set $A:=X^{(d-1)}$.
By the induction hypothesis, the Chern character is an isomorphism on $A$ (after tensoring with $R$).
Applying Lemma~\ref{fivech} yields the desired isomorphisms for $X$.
\end{proof}

\section{\texorpdfstring{Split degeneration of the crystalline-to-$\TP$ spectral sequence}{Split degeneration of the crystalline-to-TP spectral sequence}}
We fix the notation used throughout this section.  Let $p$ be a prime, and
let $K$ be a complete discretely valued nonarchimedean extension of $\Q_p$
with perfect residue field $k$.  Let $\sO_K$ be its ring of integers, and
let $\pi\in\sO_K$ be a uniformizer with ramification index $e$.  We write
$\sC$ for the completed algebraic closure $\widehat{\overline K}$, endowed
with the unique absolute value extending that of $K$.  Let $W=W(k)$, let
$K_0=\mathrm{Frac}(W)$, and put $\ol W:=W(\ol k)$. Fix a compatible system $(\zeta_n)_{n\geq 0}$ with $\zeta_0=1$ and
$\zeta_{n+1}^p=\zeta_n$, put
\[
  \epsilon=(\zeta_0,\zeta_1,\zeta_2,\ldots)
  \in \lim_{\mathrm{Frob}}\sO_\sC/p,
\]
and let $[\epsilon]\in\Ainf$ be its Teichm\"uller lift.  Let
$\mS=W[[z]]$, and let $\tilde\theta:\mS\to\sO_K$ be the usual map whose
kernel is generated by the Eisenstein polynomial $E=E(z)$ of $\pi$.  Let
$\phi:\mS\to\Ainf$ be the $W$-linear map sending $z$ to $[\pi^\flat]$, and
put
\[
  \xi:=\phi(E),
  \qquad
  \mu:=[\epsilon]-1.
\]
Finally, let $\varphi:\mS\to\mS$ be the Frobenius extending the Frobenius on
$W$ and satisfying $\varphi(z)=z^p$, let $\varphi_{\Ainf}$ be the Frobenius
on $\Ainf$, and put $\tilde\xi:=\varphi_{\Ainf}(\xi)$. For a spectrum $S$, write $\LK S$ for its $K(1)$-localization at $p$. 

\subsection{\texorpdfstring{Non-commutative $p$-adic Hodge theory}{Non-commutative p-adic Hodge theory}}
We begin by recording the coefficient rings of topological periodic homology that will be used throughout this subsection. There are canonical isomorphisms
\[
\pi_*\TP(\ol{k};\Z_p)\simeq \ol{W}[u^{\pm 1}],
\]
\[
\pi_*\TP(k;\Z_p)\simeq W[u^{\pm 1}]
\]
and
\[
\pi_*\TP(\sO_{\sC};\Z_p)\simeq \Ainf[u^{\pm 1}].
\]
Let $\sT$ be an idempotent-complete small smooth proper $\sO_K$-linear dg category. In \cite{M26}, the author proved the following.

\begin{thm}[{\cite[Theorem 1.9]{M26}}]\label{ncbms}
There exists an integer $n$ such that the following hold.
\begin{itemize}
  \item[(1)] For every $i\geq n$, the $\mS$-module $\pi_i \TCn(\sT/\bS[z];\Z_p)$ carries a natural structure of a Breuil--Kisin module.

  \item[(2)] \textup{($K(1)$-local $K$-theory comparison)} Assume that $\sT_{\sC}$ admits a geometric realization. Then for every $i\geq n$, after scalar extension along $\ol{\phi}:\mS\to \Ainf$, which is Frobenius on $W$ and sends $z$ to $[\pi^\flat]^p$, one recovers the $K(1)$-local $K$-theory of the generic fibre:
  \[
  \pi_i \TCn(\sT/\bS[z];\Z_p)\otimes_{\mS,\ol{\phi}} \Ainf[\tfrac{1}{\mu}]^{\wedge}_p
  \simeq
  \pi_i\LK K(\sT_{\sC})\otimes_{\Z_p} \Ainf[\tfrac{1}{\mu}]^{\wedge}_p.
  \]

  \item[(3)] \textup{(Topological periodic homology comparison)} For every $i\geq n$, after scalar extension along $\tilde{\phi}:\mS\to W$, which is Frobenius on $W$ and sends $z$ to $0$, one recovers the topological periodic homology of the special fibre:
  \[
  \pi_i \TCn(\sT/\bS[z];\Z_p)[\tfrac{1}{p}] \otimes_{\mS[\tfrac{1}{p}],\tilde{\phi}} K_0
  \simeq
  \pi_i \TP(\sT_k;\Z_p)[\tfrac{1}{p}].
  \]
\end{itemize}
\end{thm}

\begin{lemma}\label{tpinj}
Let $\sM$ be a $\TP(\sO_\sC;\Z_p)$-module such that $\pi_j(\sM[1/p])$ is a free $\Ainf[1/p]$-module for every $j$. Then the natural map
\[
\pi_j(\sM)\otimes_{\Ainf}\ol{W}
\longrightarrow
\pi_j\bigl(\sM\otimes^\L_{\TP(\sO_\sC;\Z_p)} \TP(\ol{k};\Z_p)\bigr)
\]
is injective, and becomes an isomorphism after inverting $p$.
\end{lemma}

\begin{proof}
For brevity, write $E := \TP(\sO_\sC;\Z_p)$ and $T := \TP(\ol{k};\Z_p)$. Choose an element $a\in\sO_\sC^\flat$ with $0<|a|<1$, for example $a=\pi^\flat$, and let $x=[a]\in\Ainf$ be its Teichm\"uller lift.  Then $x$ is a nonzero nonunit of $\Ainf$, and its image in $\ol W=W(\ol k)$ is zero, since $a$ has positive valuation.  Set $\widetilde{W}:= \colim_n \Ainf/(x^{1/p^n})$. Consider the $E$-module $\sW:= \hocolim_{n} E/(x^{1/p^n})$. Then
\[
\pi_i \sW \simeq
\begin{cases}
\widetilde{W} & \text{if $i$ is even,}\\
0 & \text{if $i$ is odd.}
\end{cases}
\]

We first show that the natural map
\[
\pi_j(\sM)\otimes_{\Ainf}\widetilde{W}
\longrightarrow
\pi_j\bigl(\sM\otimes_E^\L \sW\bigr)
\]
is injective. Put $E_n:=E/(x^{1/p^n})$. For each $n$, the cofibre sequence
\[
\sM \xrightarrow{x^{1/p^n}} \sM \to \sM\otimes_E^\L E_n
\]
induces an exact sequence
\[
\pi_j(\sM)\xrightarrow{x^{1/p^n}}\pi_j(\sM)\to
\pi_j(\sM\otimes_E^\L E_n)\to
\pi_{j-1}(\sM)\xrightarrow{x^{1/p^n}}\pi_{j-1}(\sM).
\]
It follows that the induced map
\[
\pi_j(\sM)/x^{1/p^n}\pi_j(\sM)
\longrightarrow
\pi_j(\sM\otimes_E^\L E_n)
\]
is injective. Since
\[
\pi_j(\sM)/x^{1/p^n}\pi_j(\sM)
\simeq
\pi_j(\sM)\otimes_{\Ainf}\Ainf/(x^{1/p^n}),
\]
we obtain an injective map
\[
\pi_j(\sM)\otimes_{\Ainf}\Ainf/(x^{1/p^n})
\longrightarrow
\pi_j(\sM\otimes_E^\L E_n)
\]
for every $n$. Passing to the filtered colimit over $n$, we obtain the claim.

Since the image of $x$ in $\ol{W}$ is zero, the map $E=\TP(\sO_{\sC};\Z_p) \to \TP(\ol{k};\Z_p)=T$ factors through each quotient $E_n$, and passing to the filtered homotopy colimit yields a natural $E$-linear map $\sW \to T$. Let $\sQ := \fib(\sW \to T)$. Then
\[
\pi_i \sQ \simeq
\begin{cases}
Q & \text{if $i$ is even,}\\
0 & \text{if $i$ is odd,}
\end{cases}
\qquad
Q:=\ker(\widetilde{W}\to \ol{W}).
\]
By \cite[Lemma 4.16]{BMS18}, $p$ is invertible in $Q$. It follows that $\sQ$ is naturally an $E[1/p]$-module. Since each $\pi_j(\sM[1/p])$ is free over $\Ainf[1/p]$, we obtain
\[
\pi_j(\sM)\otimes_{\Ainf} Q
\simeq
\pi_j(\sM\otimes^\L_E \sQ).
\]
We therefore obtain a commutative diagram
\[
\xymatrix{
\pi_j(\sM)\otimes_{\Ainf} Q \ar[r] \ar[d]_{a}^{\sim}
&
\pi_j(\sM)\otimes_{\Ainf} \widetilde{W} \ar[r] \ar[d]_{b}
&
\pi_j(\sM)\otimes_{\Ainf} \ol{W} \ar[d] \ar[r]
&
0
\\
\pi_j(\sM\otimes^\L_E \sQ) \ar[r]
&
\pi_j(\sM\otimes^\L_E \sW) \ar[r]
&
\pi_j(\sM\otimes^\L_E T),&
}
\]
in which both rows are exact, $a$ is an isomorphism, and $b$ is injective. This proves the lemma.
\end{proof}

\begin{prop}\label{tplength}
Let $\sT$ be an idempotent-complete small smooth proper $\sO_K$-linear dg category, and assume that $\sT_\sC$ admits a geometric realization. Then for every integer $n>0$, one has
\[
\len_{W}\bigl(\TP_i(\sT_k;\Z_p)/p^n \bigr)
\geq
\len_{\Z_p}\bigl( \LK K_i(\sT_{\sC})/p^n \bigr).
\]
\end{prop}

\begin{proof}
Since both $\TP(\sT_k;\Z_p)$ and $\LK K(\sT_{\sC})$ are $2$-periodic, we may assume that $i$ is sufficiently large. The natural surjection $\sO_\sC \to \ol{k}$ induces a morphism
\begin{equation}\label{tpOCk}
\TP(\sT_{\sO_{\sC}};\Z_p)\otimes^{\L}_{\TP(\sO_{\sC};\Z_p)} \TP(\ol{k};\Z_p)
\longrightarrow
\TP(\sT_{\ol{k}};\Z_p).
\end{equation}
By \cite[Proposition 2.12]{M26} and \cite[Proposition 3.2]{M26}, both sides of \eqref{tpOCk} define symmetric monoidal functors from the category of smooth proper $\sO_K$-linear categories to $\text{Mod}_{\TP(\ol{k};\Z_p)}(\Sp)$, and \eqref{tpOCk} defines a symmetric monoidal natural transformation between them. It follows that \eqref{tpOCk} is an equivalence (see \cite[Proposition 4.6]{AMN18}).

For brevity, write $\sM:=\TP(\sT_{\sO_{\sC}};\Z_p)$.  The freeness hypothesis in Lemma~\ref{tpinj} is satisfied for this $\sM$ after replacing $i$ by a sufficiently large integer using $2$-periodicity, by Theorem~\ref{ncbms}.  Hence, by Lemma~\ref{tpinj}, the natural morphism
\[
\pi_i(\sM)\otimes_{\Ainf}\ol{W}
\longrightarrow
\pi_i\bigl(\sM\otimes^\L_{\TP(\sO_\sC;\Z_p)} \TP(\ol{k};\Z_p)\bigr)
\xrightarrow[\sim]{\eqref{tpOCk}}
\pi_i \TP(\sT_{\ol{k}};\Z_p)
\]
is injective, and its cokernel is killed by a power of $p$. In particular,
\[
\len_{\ol{W}}\bigl(\pi_i \TP(\sT_{\ol{k}};\Z_p)/p^n \bigr)
\geq
\len_{\ol{W}}\bigl((\pi_i(\sM)\otimes_{\Ainf}\ol{W})/p^n \bigr).
\]
Combining \cite[Corollary 4.15]{BMS18}, Theorem~\ref{ncbms}, and \cite[Theorem 3.5]{M26}, we obtain
\begin{equation}\label{lenKcmp}
\len_{\ol{W}}\bigl(\pi_i \TP(\sT_{\ol{k}};\Z_p)/p^n \bigr)
\geq
\len_{\Z_p}\bigl( \LK K_i(\sT_{\sC})/p^n \bigr).
\end{equation}

Next, the base-change map $k \to \ol{k}$ induces a morphism
\begin{equation}\label{tpkolk}
\TP(\sT_k;\Z_p)\otimes^\L_{\TP(k;\Z_p)}\TP(\ol{k};\Z_p)
\longrightarrow
\TP(\sT_{\ol{k}};\Z_p).
\end{equation}
By the same argument as above, \eqref{tpkolk} is an equivalence. Since $W\to \ol{W}$ is flat, it follows that
\[
\pi_i \TP(\sT_k;\Z_p)\otimes_W \ol{W}
\simeq
\pi_i \TP(\sT_{\ol{k}};\Z_p).
\]
Hence
\[
\len_{W}\bigl(\TP_i(\sT_k;\Z_p)/p^n\bigr)
=
\len_{\ol{W}}\bigl(\pi_i \TP(\sT_{\ol{k}};\Z_p)/p^n\bigr).
\]
Combining this with \eqref{lenKcmp}, we obtain the claim.
\end{proof}
The following is the de Rham analogue of Proposition~\ref{tplength}.

\begin{prop}\label{hplength}
Let $\sT$ be an idempotent-complete small smooth proper $\sO_K$-linear dg category, and assume that $\sT_\sC$ admits a geometric realization. Then for every integer $n>0$, one has
\[
\frac{1}{e}\cdot \len_{\sO_K}\bigl(\HP_i(\sT/\sO_K)/p^n\bigr)
\geq
\len_{\Z_p}\bigl(\LK K_i(\sT_\sC)/p^n\bigr).
\]
\end{prop}

\begin{proof}
The natural map $\bS \to \sO_\sC$ induces a morphism
\begin{equation}\label{hpOCcmp}
\TP(\sT_{\sO_\sC};\Z_p)\otimes^{\L}_{\TP(\sO_\sC;\Z_p)} \HP(\sO_\sC/\sO_\sC;\Z_p)
\longrightarrow
\HP(\sT_{\sO_\sC}/\sO_\sC;\Z_p).
\end{equation}
By \cite[Proposition 2.12]{M26}, the left-hand side of \eqref{hpOCcmp} defines a symmetric monoidal functor from the category of smooth proper $\sO_K$-linear dg categories to $\Mod_{\HP(\sO_K/\sO_K;\Z_p)}(\Sp)$. On the other hand, for any idempotent-complete small smooth proper $\sO_K$-linear dg category $\sD$, the object $\HH(\sD/\sO_\sC;\Z_p)$ is dualizable in $\Mod_{H\sO_\sC}(\Sp^{BS^1})$, and by \cite[Theorem 2.15]{AMN18}, it is in fact perfect. It follows that the right-hand side of \eqref{hpOCcmp} also defines a symmetric monoidal functor from the category of smooth proper $\sO_K$-linear dg categories to $\Mod_{\HP(\sO_K/\sO_K;\Z_p)}(\Sp)$; see \cite{AMN18}. Thus, \eqref{hpOCcmp} defines a symmetric monoidal natural transformation between these two functors. Hence \eqref{hpOCcmp} is an equivalence by \cite[Proposition 4.6]{AMN18}.

We recall that $\xi$ generates the kernel of the map $\Ainf \to \sO_\sC$, and that
\[
\HP(\sO_\sC/\sO_\sC;\Z_p)
\simeq
\cofib\bigl(\TP(\sO_\sC;\Z_p)\xrightarrow{\cdot \xi}\TP(\sO_\sC;\Z_p)\bigr)
=
\TP(\sO_\sC;\Z_p)/(\xi).
\]
Using the equivalence \eqref{hpOCcmp} and the basic properties of the derived tensor product in module spectra, we obtain a short exact sequence
\begin{equation}\label{hpseq}
0
\to
\TP_i(\sT_{\sO_\sC};\Z_p)\otimes_{\Ainf}\sO_\sC
\to
\HP_i(\sT_{\sO_\sC}/\sO_\sC;\Z_p)
\to
\TP_{i-1}(\sT_{\sO_\sC};\Z_p)[\xi]
\to
0.
\end{equation}

We now use the discussion in \cite[\S 7.10]{CK19}. Normalize the valuation $\val$ on $\sO_\sC$ by the condition $\val(p)=1$. By the structure theorem, every finitely presented $\sO_\sC$-module $N$ is of the form
\[
N \simeq \bigoplus_{j=1}^m \sO_\sC/(a_j),
\qquad a_j\in \sO_\sC.
\]
If $N$ is torsion, we define
\[
\val(N):=\sum_{j=1}^m \val(a_j).
\]
By \cite[Lemma 7.11]{CK19}, one has
\[
\len_{W(\sC^\flat)}\bigl((\TP_i(\sT_{\sO_\sC};\Z_p)/p^n)\otimes_{\Ainf} W(\sC^\flat)\bigr)
\leq
\val\bigl((\TP_i(\sT_{\sO_\sC};\Z_p)/p^n)\otimes_{\Ainf}\sO_\sC\bigr).
\]
Since $\TP_{i-1}(\sT_{\sO_\sC};\Z_p)[\xi]$ is a torsion $\sO_\sC$-module, we have
\[
\val\bigl(\TP_{i-1}(\sT_{\sO_\sC};\Z_p)[\xi]/p^n\bigr)
=
\val\bigl(\Tor_1^{\sO_\sC}(\TP_{i-1}(\sT_{\sO_\sC};\Z_p)[\xi],\sO_\sC/p^n)\bigr).
\]
Therefore, by \eqref{hpseq},
\[
\val\bigl((\TP_i(\sT_{\sO_\sC};\Z_p)/p^n)\otimes_{\Ainf}\sO_\sC\bigr)
\leq
\val\bigl(\HP_i(\sT_{\sO_\sC}/\sO_\sC;\Z_p)/p^n\bigr).
\]
Using Theorem~\ref{ncbms}(2), namely
\[
(\TP_i(\sT_{\sO_\sC};\Z_p)/p^n)\otimes_{\Ainf} W(\sC^\flat)
\simeq
\LK K_i(\sT_\sC)/p^n \otimes_{\Z_p} W(\sC^\flat),
\]
together with the base change isomorphism
\[
\HP_i(\sT_{\sO_\sC}/\sO_\sC;\Z_p)
\simeq
\HP_i(\sT/\sO_K;\Z_p)\otimes_{\sO_K}\sO_\sC,
\]
we obtain the desired inequality.
\end{proof}
\subsection{Preliminaries on spectral sequences I}
The aim of this subsection is to give spectral analogues of Li's definitions and
criteria for filtered complexes over a DVR \cite[Section~2]{L22}. More precisely, we introduce the
notions of degeneration, saturated degeneration, and split degeneration for a
filtered spectrum, and then prove the corresponding torsion-length criterion. Let $R$ be a DVR with uniformizer $\varpi$, fraction field
$K=R[1/\varpi]$, and residue field $\kappa=R/\varpi$.
For a finitely generated $R$-module $N$, write
\begin{align*}
  N_{\tors} &\subset N, \\
  N_{\tf} &:= N/N_{\tors}
\end{align*}
for its torsion submodule and torsion-free quotient, respectively.

Let $M$ be a spectrum equipped with a decreasing filtration
\begin{equation*}
  \cdots \longrightarrow \fil^{n+1}M
  \longrightarrow \fil^n M
  \longrightarrow \fil^{n-1}M
  \longrightarrow \cdots
  \longrightarrow M .
\end{equation*}
Set
\begin{equation*}
  \gr^n M :=
  \operatorname{cofib}\bigl(\fil^{n+1}M\to \fil^n M\bigr).
\end{equation*}
We assume throughout that the filtration is complete and exhaustive. We also assume that the induced spectral sequence
\begin{equation*}
  E_1^{n,*}=\pi_*(\gr^n M)\Longrightarrow \pi_*(M)
\end{equation*}
is strongly convergent and locally finite in each homotopy degree. More
explicitly, for every $i$, only finitely many $n$ contribute to the
filtration on $\pi_i(M)$. Finally, assume that all groups $\pi_i(M),\pi_i(\fil^nM), \pi_i(\gr^nM)$ are finitely generated $R$-modules, and all maps appearing in the exact
couple are $R$-linear.

For a torsion-free inclusion $A\subset B$ of finitely generated
$R$-modules, we say that $A\subset B$ is saturated if
\begin{equation*}
  \varpi A=A\cap \varpi B,
\end{equation*}
or equivalently if $B/A$ is torsion-free.

\begin{definition}[{Spectral version of \cite[Definition 2.3]{L22}}]
Let $M$ be as above.

\begin{enumerate}
\item We say that the filtered spectrum $M$, or equivalently its associated
spectral sequence, is \emph{degenerate} if for all $i,n\in \mathbb Z$, the
natural map
\begin{equation*}
  \pi_i(\fil^nM)\longrightarrow \pi_i(M)
\end{equation*}
is injective.

\item We say that $M$ is \emph{saturated degenerate} if it is degenerate
and, for all $i,n\in \mathbb Z$, the induced injection
\begin{equation*}
  \pi_i(\fil^nM)_{\tf}\longrightarrow \pi_i(M)_{\tf}
\end{equation*}
is saturated.

\item We say that $M$ is \emph{split degenerate} if it is degenerate and,
for all $i,n\in \mathbb Z$, the injection
\begin{equation*}
  \pi_i(\fil^nM)\longrightarrow \pi_i(M)
\end{equation*}
splits as a map of $R$-modules.
\end{enumerate}
\end{definition}

\begin{proposition}[{Spectral version of \cite[Proposition 2.5]{L22}}]\label{sscrit}
Assume that the spectral sequence associated with $M$ degenerates after
inverting $\varpi$, i.e. after tensoring all $E_r$-terms with $K$.
Then the following hold.

\begin{enumerate}
\item The filtered spectrum $M$ is saturated degenerate if and only if
\begin{equation*}
  \len_R\bigl(\pi_i(M)_{\tors}\bigr)
  =
  \sum_n
  \len_R\bigl(\pi_i(\gr^nM)_{\tors}\bigr)
\end{equation*}
for every $i\in\mathbb Z$.

\item The filtered spectrum $M$ is split degenerate if and only if there is
an abstract isomorphism of $R$-modules
\begin{equation*}
  \pi_i(M)_{\tors}
  \simeq
  \bigoplus_n \pi_i(\gr^nM)_{\tors}
\end{equation*}
for every $i\in\mathbb Z$.
\end{enumerate}
\end{proposition}

\begin{proof}
Let
\begin{equation*}
  F^n\pi_i(M):=
  \operatorname{im}\bigl(\pi_i(\fil^nM)\to \pi_i(M)\bigr)
\end{equation*}
be the filtration induced on $\pi_i(M)$. Since the spectral sequence is
strongly convergent and locally finite in each degree, its $E_\infty$-page
identifies with the associated graded pieces
\begin{equation*}
  E_\infty^{n,i}\simeq \gr_F^n\pi_i(M).
\end{equation*}

Because the spectral sequence degenerates after inverting $\varpi$, every
differential is $\varpi$-power torsion. In particular,
$(\gr_F^n\pi_i(M))_{\tors}$ is a subquotient of
$\pi_i(\gr^nM)_{\tors}$. Hence
\begin{align*}
  \len_R\bigl(\pi_i(M)_{\tors}\bigr)
  &\le
  \sum_n
  \len_R\bigl((\gr_F^n\pi_i(M))_{\tors}\bigr) \\
  &\le
  \sum_n
  \len_R\bigl(\pi_i(\gr^nM)_{\tors}\bigr).
\end{align*}
If equality with the right-hand side holds, then all torsion classes on the
$E_1$-page survive to $E_\infty$. Since all differentials are torsion after
inverting $\varpi$, this forces every differential to vanish. Thus the
spectral sequence is degenerate.

Assume now that the spectral sequence is degenerate. Then
\begin{align*}
  F^n\pi_i(M)
  &=
  \pi_i(\fil^nM) \\
  &\subset
  \pi_i(M),
\end{align*}
and the associated graded pieces of this filtration are
\begin{equation*}
  \gr_F^n\pi_i(M)\simeq \pi_i(\gr^nM).
\end{equation*}
For a finitely generated $R$-module $N$ with a finite exhaustive filtration
$F^* N$, by \cite[Lemma 2.6]{L22}, 
\begin{equation*}
  F^nN_{\tf}\subset N_{\tf}
  \text{ is saturated for all } n
\end{equation*}
if and only if
\begin{equation*}
  \len_R(N_{\tors})
  =
  \sum_n \len_R\bigl((\gr_F^nN)_{\tors}\bigr).
\end{equation*}
Applying this lemma to $N=\pi_i(M)$ gives the first assertion.

For the second assertion, first suppose that $M$ is split degenerate.
Then each inclusion
\begin{equation*}
  \pi_i(\fil^nM)\subset \pi_i(M)
\end{equation*}
splits, so the induced filtration on $\pi_i(M)_{\tors}$ is split. Since
its graded pieces are $\pi_i(\gr^nM)_{\tors}$, we obtain an abstract
isomorphism
\begin{equation*}
  \pi_i(M)_{\tors}
  \simeq
  \bigoplus_n \pi_i(\gr^nM)_{\tors}.
\end{equation*}

Conversely, suppose that such an abstract isomorphism exists for every $i$.
Taking lengths gives the equality in part $(1)$, hence $M$ is saturated
degenerate. Therefore
\begin{equation*}
  \pi_i(\fil^nM)_{\tf}\subset \pi_i(M)_{\tf}
\end{equation*}
is saturated, and hence split, since finitely generated torsion-free modules
over a DVR are free. It remains to split the torsion part. The induced
filtration on $\pi_i(M)_{\tors}$ has graded pieces
\begin{equation*}
  \pi_i(\gr^nM)_{\tors},
\end{equation*}
and by the abstract isomorphism
\begin{equation*}
  \pi_i(M)_{\tors}
  \simeq
  \bigoplus_n \pi_i(\gr^nM)_{\tors}.
\end{equation*}
By \cite[Corollary 2.11]{L22} each inclusion
\begin{equation*}
  \pi_i(\fil^nM)_{\tors}
  \subset
  \pi_i(M)_{\tors}
\end{equation*}
splits. Combining the splitting on torsion parts with the splitting on
torsion-free quotients gives a splitting
\begin{equation*}
  \pi_i(\fil^nM)\longrightarrow \pi_i(M)
\end{equation*}
for every $i,n$. Thus $M$ is split degenerate.
\end{proof}

\begin{lemma}\label{lenfil}
Let $M$ be a filtered spectrum as above, and assume that $M$ is saturated degenerate.
Then for every $i\in \mathbb Z$ and every integer $n>0$, one has
\begin{equation*}
  \len_R\bigl(\pi_i(M)/\varpi^n\bigr)
  \leq
  \sum_a \len_R\bigl(\pi_i(\gr^a M)/\varpi^n\bigr).
\end{equation*}
\end{lemma}

\begin{proof}
Set
\begin{equation*}
  N:=\pi_i(M),
\end{equation*}
and let
\begin{equation*}
  F^aN:=
  \operatorname{im}\bigl(\pi_i(\fil^a M)\to \pi_i(M)\bigr)
\end{equation*}
be the induced filtration on $N$.
Since the filtration is locally finite in each degree, this is a finite
exhaustive filtration on the finitely generated $R$-module $N$.

Fix $n>0$. The filtration $F^* N$ induces a filtration on
$N/\varpi^nN$ by
\begin{equation*}
  F^a(N/\varpi^nN):=
  \operatorname{im}\bigl(F^aN\to N/\varpi^nN\bigr).
\end{equation*}
Hence
\begin{equation*}
  \len_R\bigl(N/\varpi^nN\bigr)
  =
  \sum_a
  \len_R\bigl(\gr_F^a(N/\varpi^nN)\bigr).
\end{equation*}

For each $a$, there is a natural surjection
\begin{equation*}
  \gr_F^a N/\varpi^n \gr_F^a N
  \twoheadrightarrow
  \gr_F^a(N/\varpi^nN).
\end{equation*}
Therefore
\begin{equation*}
  \len_R\bigl(\gr_F^a(N/\varpi^nN)\bigr)
  \leq
  \len_R\bigl(\gr_F^a N/\varpi^n \gr_F^a N\bigr).
\end{equation*}
Summing over $a$, we obtain
\begin{equation*}
  \len_R\bigl(N/\varpi^nN\bigr)
  \leq
  \sum_a
  \len_R\bigl(\gr_F^a N/\varpi^n \gr_F^a N\bigr).
\end{equation*}

Since $M$ is degenerate, we have
\begin{equation*}
  F^aN=\pi_i(\fil^aM)
\end{equation*}
and hence
\begin{equation*}
  \gr_F^aN \simeq \pi_i(\gr^aM).
\end{equation*}
Thus
\begin{equation*}
  \len_R\bigl(\pi_i(M)/\varpi^n\bigr)
  \leq
  \sum_a
  \len_R\bigl(\pi_i(\gr^aM)/\varpi^n\bigr),
\end{equation*}
as claimed.
\end{proof}
\subsection{\texorpdfstring{Split degeneration of the crystalline-to-$\TP$ spectral sequence}{Split degeneration of the crystalline-to-TP spectral sequence}}

Let $X$ be a smooth proper variety over $\sO_K$. Applying Proposition~\ref{tplength} to $\sT=\perf(X)$, we obtain the inequality
\[
\len_{W}\bigl(\TP_i(X_k;\Z_p)/p^n\bigr)
\geq
\len_{\Z_p}\bigl(\LK K_i(X_\sC)/p^n\bigr).
\]
By Theorem~\ref{ncbms}, the groups $\TP_i(X_k;\Z_p)$ and $\LK K_i(X_\sC)$ have the same rank, say $r$. Hence, for $n \gg 0$, we may write
\[
\TP_i(X_k;\Z_p)/p^n
\simeq
\TP_i(X_k;\Z_p)_{\tors} \oplus W^{\oplus r}/p^n
\]
and
\[
\LK K_i(X_\sC)/p^n
\simeq
\LK K_i(X_\sC)_{\tors} \oplus \Z_p^{\oplus r}/p^n.
\]
It follows that
\begin{equation}\label{ineqtp}
\len_{W}\bigl(\TP_i(X_k;\Z_p)_{\tors}\bigr)
\geq
\len_{\Z_p}\bigl(\LK K_i(X_\sC)_{\tors}\bigr).
\end{equation}

On the other hand, Thomason studied the \'etale sheafification of the Postnikov tower of $\LK K(-)$ in \cite{T85}, and obtained a strongly convergent spectral sequence
\[
E_2^{i,j}=H_{\et}^{i-j}(X_\sC,\Z_p(j)) \Longrightarrow \pi_{-i-j}\bigl(\LK K(X_\sC)\bigr),
\]
which degenerates after tensoring with $\Q_p$. Therefore,
\begin{equation}\label{ineqahss}
\sum_{n\in \Z}\len_{\Z_p}\bigl(H_{\et}^{i+2n}(X_\sC,\Z_p)_{\tors}\bigr)
\geq
\len_{\Z_p}\bigl(\LK K_i(X_\sC)_{\tors}\bigr).
\end{equation}

Moreover, in \cite{BMS19}, Bhatt--Morrow--Scholze studied the quasisyntomic sheafification of the Postnikov tower of $\TP(-;\Z_p)$ and obtained a strongly convergent spectral sequence
\[
E_2^{i,j}=H_{\cry}^{i-j}(X_k/W)(j)\Longrightarrow \pi_{-i-j}\bigl(\TP(X_k;\Z_p)\bigr),
\]
which degenerates after inverting $p$; see also \cite{E18}. Hence
\begin{equation}\label{ineqcry}
\sum_{n\in \Z}\len_{W}\bigl(H_{\cry}^{i+2n}(X_k/W)_{\tors}\bigr)
\geq
\len_{W}\bigl(\TP_i(X_k;\Z_p)_{\tors}\bigr).
\end{equation}

Combining \eqref{ineqtp}, \eqref{ineqahss}, and \eqref{ineqcry} with \cite[Theorem 14.5]{BMS18}, we obtain the following diagram of inequalities:
\begin{equation}\label{ineqdia}
\xymatrix{
\sum_{n\in \Z}\len_{W}\bigl(H_{\cry}^{i+2n}(X_k/W)_{\tors}\bigr)
  & \overset{(1)}{\geq} &
\sum_{n\in \Z}\len_{\Z_p}\bigl(H_{\et}^{i+2n}(X_\sC,\Z_p)_{\tors}\bigr) \\
\len_{W}\bigl(\TP_i(X_k;\Z_p)_{\tors}\bigr) \ar@{}[u]^{(2){\rotatebox{-90}{$\geq$}}}
  & \overset{(4)}{\geq} &
\len_{\Z_p}\bigl(\LK K_i(X_\sC)_{\tors}\bigr) \ar@{}[u]^{{\rotatebox{-90}{$\geq$}}(3)}
}
\end{equation}

The following proposition shows that, under a suitable dimension bound, the inequality $(3)$ in \eqref{ineqdia} is in fact an equality.

\begin{prop}\label{kdeg}
Let $Y$ be a smooth proper variety over $\sC$. Assume that $\dim Y<p-1$. Then the Atiyah--Hirzebruch spectral sequence is split degenerate. In particular, there is an isomorphism of $\Z_p$-modules
\begin{equation}\label{ksplit}
\LK K_i(Y) \simeq \bigoplus_{n\in\Z} H_{\et}^{i+2n}(Y,\Z_p).
\end{equation}
\end{prop}

\begin{proof}
Fix an embedding $\bC_p \hookrightarrow \sC$. There exists a subfield $\bC_p \subset L \subset \sC$ such that $L$ is finitely generated over $\bC_p$ and $Y$ is defined over $L$. Choose a variety $T$ over $\bC_p$ with function field $K(T)\simeq L$. Then, after replacing $T$ by a nonempty open subscheme, there exists a smooth projective morphism $Y' \to T$ over $\bC_p$ whose generic fibre becomes isomorphic to $Y$ after base change to $\sC$.

Choose a $\bC_p$-valued point $t\in T(\bC_p)$. By proper smooth base change in \'etale cohomology and Thomason's spectral sequence \cite{T85}, we obtain canonical isomorphisms
\[
\LK K_j(Y'_t)\simeq \LK K_j(Y)
\]
and
\[
H_{\et}^{*}(Y'_t,\Z_p)\simeq H_{\et}^{*}(Y,\Z_p).
\]
Fix an isomorphism $\bC_p \simeq \bC$. Then $(Y'_{t,\bC})^{\an}$ is homotopy equivalent to a finite CW complex. Using the comparison isomorphisms
\[
\LK K_*(Y'_t)\simeq K^{\mathrm{top}}_*((Y'_{t,\bC})^{\an})\otimes_{\Z}\Z_p
\]
and
\[
H_{\et}^{*}(Y'_t,\Z_p)\simeq H_{\Sing}^{*}(Y'_{t,\bC},\Z_p),
\]
together with Theorem~\ref{ktopchdec}, we obtain \eqref{ksplit}. By Proposition~\ref{sscrit}, the claim follows.
\end{proof}

\begin{thm}\label{crydege}
Let $X$ be a smooth proper variety over $\sO_K$. Assume that $\dim X<p-1$ if $\sO_K=W$, and assume that $e(2\dim X-1)<p-1$ if $e\geq 2$. Then the crystalline-to-$\TP$ spectral sequence is split degenerate. In particular, there is an isomorphism of $W$-modules
\[
\pi_j\TP(X_k;\Z_p) \simeq \bigoplus_{i\in\Z} H_{\cry}^{-j+2i}(X_k/W).
\]
\end{thm}

\begin{proof}
The integral comparison theorem between crystalline and \'etale cohomology applies; see \cite{FM87} or \cite[Theorem 4.12]{M21}. Thus
\[
H^*_{\cry}(X_k/W) \simeq H^*_{\et}(X_\sC,\Z_p)\otimes_{\Z_p}W.
\]
It follows that the inequality $(1)$ in \eqref{ineqdia} is in fact an equality.

Hence the inequality $(2)$ in \eqref{ineqdia} is also an equality. By Proposition~\ref{sscrit}, it follows that the strongly convergent spectral sequence
\begin{equation}\label{satcry}
E_2^{a,2b}=H_{\cry}^{a}(X_k/W)(b)\Longrightarrow \pi_{2b-a}\bigl(\TP(X_k;\Z_p)\bigr)
\end{equation}
is saturated degenerate. Therefore, for every integer $n>0$, we obtain
\[
\xymatrix{
\sum_{m\in \Z}\len_{W}\bigl(H_{\cry}^{i+2m}(X_k/W)/p^n\bigr)
  & \overset{(1)'}{=} &
\sum_{m\in \Z}\len_{\Z_p}\bigl(H_{\et}^{i+2m}(X_\sC,\Z_p)/p^n\bigr) \\
\len_{W}\bigl(\TP_i(X_k;\Z_p)/p^n\bigr) \ar@{}[u]^{(2)'{\rotatebox{-90}{$\geq$}}}
  & \overset{(4)'}{\geq} &
\len_{\Z_p}\bigl(\LK K_i(X_\sC)/p^n\bigr) \ar@{}[u]^{{\rotatebox{-90}{$=$}}(3)'}
}
\]
Here, $(2)'$ follows from Lemma~\ref{lenfil} together with the saturated degeneracy of \eqref{satcry}, $(4)'$ follows from Proposition~\ref{tplength}, $(1)'$ follows from \cite[Theorems 4.12 and 5.13]{M21}, and $(3)'$ follows from Proposition~\ref{kdeg}. Thus $(2)'$ is also an equality. By the structure theorem for finitely generated $W$-modules and Proposition~\ref{sscrit}, the claim follows.
\end{proof}

\begin{remark}
If $\sO_K=W$, then by the comparison theorem between crystalline and de Rham cohomology \cite[Theorem V.2.3.2]{B74}, together with the comparison theorem between $\TP$ and $\HP$ proved independently by Scholze (unpublished), Petrov--Vologodsky \cite{PV19arxiv}, Devalapurkar--Raksit \cite{DR25}, and Mao \cite{M25}, Theorem~\ref{crydege} implies that there is an isomorphism of $W$-modules
\[
\pi_j\HP(X/W) \simeq \bigoplus_{i\in\Z} H_{\dR}^{-j+2i}(X/W).
\]
\end{remark}


\section{Preliminaries on Breuil--Kisin modules}

In this section we recall the Breuil--Kisin input used later.  We first fix
our conventions and the low-ramification structure theorem for torsion
Breuil--Kisin modules.  We then record a slightly enlarged version of the
structure theorem, allowing finite free Breuil--Kisin modules in addition to
$p$-power torsion ones.

\subsection{Definitions and structure theorems}

Let $\mS=W(k)[[z]]$, equipped with the Frobenius $\varphi$ extending the
Witt-vector Frobenius and satisfying $\varphi(z)=z^p$.  Let $E=E(z)$ be the
Eisenstein polynomial of the chosen uniformizer $\pi$ of $\sO_K$, so that
$\mS/(E)\simeq \sO_K$.

\begin{definition}
\begin{itemize}
  \item[(1)] A \emph{Breuil--Kisin module} is a finitely generated
  $\mS$-module $\fM$ together with an isomorphism
  \[
    \Phi_{\fM}:\varphi^*\fM[1/E]
    \xrightarrow{\sim}
    \fM[1/E],
    \qquad
    \varphi^*\fM:=\fM\otimes_{\mS,\varphi}\mS .
  \]

  \item[(2)] A \emph{Breuil--Kisin--Fargues module} is a finitely generated
  $\Ainf$-module $\widehat{\fM}$ together with an isomorphism
  \[
    \Phi_{\widehat{\fM}}:
    \varphi_{\Ainf}^*\widehat{\fM}[1/\tilde{\xi}]
    \xrightarrow{\sim}
    \widehat{\fM}[1/\tilde{\xi}] .
  \]

  \item[(3)] Let $(\fM,\Phi_{\fM})$ be a Breuil--Kisin module, and view
  $\Phi_{\fM}$ as an $\mS$-linear map
  \[
    \Phi_{\fM}:\varphi^*\fM\longrightarrow \fM[1/E].
  \]
  For integers $r\geq s$, we say that the $E$-height of $\fM$ is contained
  in $[s,r]$ if
  \[
    E^r\fM\subset \Im(\Phi_{\fM})\subset E^s\fM
  \]
  inside $\fM[1/E]$.
\end{itemize}
\end{definition}

For a Breuil--Kisin module $\fM$, the base change
$\fM\otimes_{\mS,\ol{\phi}}\Ainf$ carries a natural
Breuil--Kisin--Fargues module structure; see
\cite[Proposition~4.32]{BMS18}.

We also recall the standard structure theorem for Breuil--Kisin modules.

\begin{proposition}[{\cite[Proposition~4.3]{BMS18}}]\label{bkdecomp}
Let $(\fM,\Phi_{\fM})$ be a Breuil--Kisin module.  Then there is a canonical
exact sequence of Breuil--Kisin modules
\[
  0
  \to \fM_{\tors}
  \to \fM
  \to \fM_{\free}
  \to \ol{\fM}
  \to 0,
\]
where $\fM_{\tors}$ is the $p$-power torsion submodule, $\fM_{\free}$ is
finite free over $\mS$, and $\ol{\fM}$ is killed by a power of the maximal
ideal $(p,z)$.
\end{proposition}

We next recall the low-ramification structure theorem for torsion
Breuil--Kisin modules, following \cite{CL09} and \cite[Section~6]{M21}.

\begin{definition}[{\cite[Section~2.1]{CL09}}]
\begin{itemize}
  \item[(1)] Let $\BKMod^{r,\varphi}$ be the category of Breuil--Kisin
  modules whose $E$-height is contained in $[0,r]$.

  \item[(2)] Let $\BKModone^{r,\varphi}$ be the full subcategory of
  $\BKMod^{r,\varphi}$ spanned by the objects that are finite free over
  $\mS_1:=\mS/p$.

  \item[(3)] Let $\BKModinf^{r,\varphi}$ be the smallest full subcategory of
  $\BKMod^{r,\varphi}$ that contains $\BKModone^{r,\varphi}$ and is stable
  under extensions.
\end{itemize}
\end{definition}

Liu proved that an
object $\fM\in\BKMod^{r,\varphi}$ belongs to $\BKModinf^{r,\varphi}$ if and
only if its underlying $\mS$-module is $p$-power torsion and has no
$z$-torsion; see \cite[Section~2.3]{L07}.

\begin{thm}[{\cite[Theorem~6.20]{M21}}]\label{bkstruct}
Assume $er<p-1$.  For every $\fM\in\BKModinf^{r,\varphi}$, the underlying
$\mS$-module of $\fM$ is isomorphic to
\[
  \bigoplus_{i=1}^n \mS/p^{a_i}
\]
for some integers $a_i\geq 1$.
\end{thm}

\subsection{A variant allowing finite free Breuil--Kisin modules}

In this subsection we prove a variant of Theorem~\ref{bkstruct} in which
finite free Breuil--Kisin modules are allowed together with torsion ones.

\begin{definition}
Let $\overline{\Mod}^{r,\varphi}_{/\mS_1}$ be the full subcategory of
$\widetilde{\Mod}^{r,\varphi}_{/\mS}$ spanned by those objects whose underlying
$\mS$-module is either finite free over $\mS_1=\mS/p=k[[z]]$ or finite free
over $\mS$.

We define
\[
  \overline{\Mod}^{r,\varphi}_{/\mS_\infty}
\]
to be the smallest full subcategory of
\[
  \widetilde{\Mod}^{\infty,\varphi}_{/\mS}
  :=
  \bigcup_{n>0}\widetilde{\Mod}^{n,\varphi}_{/\mS}
\]
that contains $\overline{\Mod}^{r,\varphi}_{/\mS_1}$ and is stable under
extensions.
\end{definition}

\begin{remark}
The category $\overline{\Mod}^{r,\varphi}_{/\mS_\infty}$ need not be a full
subcategory of $\widetilde{\Mod}^{r,\varphi}_{/\mS}$: successive extensions may
increase the $E$-height.  Thus an object of
$\overline{\Mod}^{r,\varphi}_{/\mS_\infty}$ can have $E$-height larger than
$r$.
\end{remark}

\begin{definition}
Let $\overline{\ModFI}^{r,\varphi}_{/\mS_\infty}$ be the full subcategory of
$\overline{\Mod}^{r,\varphi}_{/\mS_\infty}$ spanned by those objects $M$
whose underlying $\mS$-module is of the form
\[
  M\simeq
  \mS^{\oplus m}
  \oplus
  \bigoplus_{i=1}^n \mS/p^{a_i}
\]
for some integers $m,n\geq 0$ and $a_i\geq 1$.
\end{definition}

\begin{definition}
Let $\mathcal C_1^{r,\varphi}$ be the full subcategory of
$\widetilde{\Mod}^{\infty,\varphi}_{/\mS}$ consisting of objects killed by $p$ which
admit a finite filtration whose graded pieces lie in
$\Mod^{r,\varphi}_{/\mS_1}$.
\end{definition}

\begin{remark}
Every object of $\mathcal C_1^{r,\varphi}$ is finite free over
$\mS_1=k[[z]]$.  Indeed, it is obtained by successive extensions of finite
free $k[[z]]$-modules, and $k[[z]]$ is a principal ideal domain.
\end{remark}

We shall use the following low-ramification abelianity statement.

\begin{lemma}[{\cite[Corollary~6.15]{M21}}]\label{bkoneab}
Assume $er<p-1$.  Then $\Mod^{r,\varphi}_{/\mS_1}$ is an abelian category.
Moreover, for every morphism in $\Mod^{r,\varphi}_{/\mS_1}$, the
categorical kernel and cokernel agree with the usual kernel and cokernel of
the underlying $\mS$-modules.
\end{lemma}

\begin{lemma}\label{cokcone}
Assume $er<p-1$.  Let $Q$ be an object of
$\Mod^{r,\varphi}_{/\mS_1}$, and let $N$ be an object of
$\mathcal C_1^{r,\varphi}$.  For any morphism
\[
  f:Q\longrightarrow N
\]
in $\widetilde{\Mod}^{\infty,\varphi}_{/\mS}$, the usual $\mS_1$-modules
$\im(f)$ and $\Coker(f)$, equipped with their induced Frobenius maps, belong
to $\mathcal C_1^{r,\varphi}$.
\end{lemma}

\begin{proof}
We argue by induction on the length of a filtration of $N$ whose graded
pieces lie in $\Mod^{r,\varphi}_{/\mS_1}$.

If $N$ lies in $\Mod^{r,\varphi}_{/\mS_1}$, then $f$ is a morphism in the
abelian category $\Mod^{r,\varphi}_{/\mS_1}$.  By Lemma~\ref{bkoneab}, the
usual image and cokernel of $f$ belong to $\Mod^{r,\varphi}_{/\mS_1}$, and
hence to $\mathcal C_1^{r,\varphi}$.

Now suppose that $N$ fits into an exact sequence
\[
  0\longrightarrow N'\longrightarrow N\longrightarrow N''\longrightarrow 0,
\]
where $N'\in\mathcal C_1^{r,\varphi}$ has smaller filtration length and
$N''\in\Mod^{r,\varphi}_{/\mS_1}$.  Let
\[
  g:Q\longrightarrow N''
\]
be the composite of $f$ with $N\to N''$.  Since $Q$ and $N''$ lie in the
abelian category $\Mod^{r,\varphi}_{/\mS_1}$, the usual $\mS_1$-modules
$\ker(g)$, $\im(g)$, and $\Coker(g)$ lie in
$\Mod^{r,\varphi}_{/\mS_1}$.

The morphism $f$ induces a morphism
\[
  \ker(g)\longrightarrow N'.
\]
By the induction hypothesis, its image and cokernel belong to
$\mathcal C_1^{r,\varphi}$.  The snake lemma gives exact sequences
\[
  0\longrightarrow
  \im\bigl(\ker(g)\to N'\bigr)
  \longrightarrow
  \im(f)
  \longrightarrow
  \im(g)
  \longrightarrow 0
\]
and
\[
  0\longrightarrow
  \Coker\bigl(\ker(g)\to N'\bigr)
  \longrightarrow
  \Coker(f)
  \longrightarrow
  \Coker(g)
  \longrightarrow 0.
\]
Since $\mathcal C_1^{r,\varphi}$ is stable under extensions, both
$\im(f)$ and $\Coker(f)$ belong to $\mathcal C_1^{r,\varphi}$.
\end{proof}

\begin{definition}
For an $\mS$-module $M$ and an integer $j\geq 0$, put
\[
  \gr_p^j M:=p^jM/p^{j+1}M.
\]
\end{definition}

\begin{lemma}\label{grpextone}
Assume $er<p-1$.  Let
\[
  0\longrightarrow A\longrightarrow B\longrightarrow Q\longrightarrow 0
\]
be an exact sequence in $\widetilde{\Mod}^{\infty,\varphi}_{/\mS}$, where
$Q\in\Mod^{r,\varphi}_{/\mS_1}$.  Suppose that
\[
  \gr_p^j A\in\mathcal C_1^{r,\varphi}
  \qquad
  \text{for all } j\geq 0.
\]
Then
\[
  \gr_p^j B\in\mathcal C_1^{r,\varphi}
  \qquad
  \text{for all } j\geq 0.
\]
\end{lemma}

\begin{proof}
For each $j\geq 0$, the extension defines a natural morphism
\[
  c_j:Q\longrightarrow \gr_p^j A.
\]
Indeed, for $q\in Q$, choose a lift $\tilde q\in B$.  Since $pQ=0$, we have
$p\tilde q\in A$.  We set
\[
  c_j(q):=p^j(p\tilde q)\bmod p^{j+1}A.
\]
This is independent of the choice of $\tilde q$, because replacing
$\tilde q$ by $\tilde q+a$ with $a\in A$ changes $p^j(p\tilde q)$ by
$p^{j+1}a$.  Since $\varphi(p)=p$, the map $c_j$ is compatible with
Frobenius.

By Lemma~\ref{cokcone}, both $\im(c_j)$ and $\Coker(c_j)$ belong to
$\mathcal C_1^{r,\varphi}$.

We claim that there are exact sequences
\[
  0\longrightarrow
  \Coker(c_0)
  \longrightarrow
  \gr_p^0 B
  \longrightarrow
  Q
  \longrightarrow 0
\]
and, for $j\geq 1$,
\[
  0\longrightarrow
  \Coker(c_j)
  \longrightarrow
  \gr_p^j B
  \longrightarrow
  \im(c_{j-1})
  \longrightarrow 0.
\]
For $j=0$, the map $\gr_p^0B=B/pB\to Q$ is induced by $B\to Q$.  Its kernel
is $A/(A\cap pB)$, and $(A\cap pB)/pA$ is exactly $\im(c_0)$.  Hence this
kernel is $\Coker(c_0)$.

For $j\geq 1$, define
\[
  \gr_p^jB\longrightarrow \gr_p^{j-1}A
\]
by sending the class of $p^jb$ to the class of $p^{j-1}(pb)$.  Its image is
$\im(c_{j-1})$.  Its kernel is the image of $\gr_p^jA$, and the kernel of
$\gr_p^jA\to\gr_p^jB$ is $\im(c_j)$.  Therefore this kernel is
$\Coker(c_j)$, proving the claimed exact sequence.

Since $\mathcal C_1^{r,\varphi}$ is stable under extensions, the displayed
exact sequences show that $\gr_p^jB$ belongs to $\mathcal C_1^{r,\varphi}$
for all $j\geq 0$.
\end{proof}

\begin{lemma}\label{grpextfree}
Let
\[
  0\longrightarrow A\longrightarrow B\longrightarrow P\longrightarrow 0
\]
be an exact sequence in $\widetilde{\Mod}^{\infty,\varphi}_{/\mS}$, where $P$ is
finite free over $\mS$ and lies in $\widetilde{\Mod}^{r,\varphi}_{/\mS}$.  Suppose
that
\[
  \gr_p^j A\in\mathcal C_1^{r,\varphi}
  \qquad
  \text{for all } j\geq 0.
\]
Then
\[
  \gr_p^j B\in\mathcal C_1^{r,\varphi}
  \qquad
  \text{for all } j\geq 0.
\]
\end{lemma}

\begin{proof}
Since $P$ is finite free over $\mS$, it is $p$-torsion-free.  Hence, for
each $j\geq 0$, the sequence
\[
  0\longrightarrow
  \gr_p^j A
  \longrightarrow
  \gr_p^j B
  \longrightarrow
  \gr_p^j P
  \longrightarrow 0
\]
is exact.  Multiplication by $p^j$ identifies $P/pP$ with $\gr_p^jP$.
Since $P$ has $E$-height contained in $[0,r]$, the reduction $P/pP$ belongs
to $\Mod^{r,\varphi}_{/\mS_1}$.  Therefore $\gr_p^jP$ belongs to
$\mathcal C_1^{r,\varphi}$ for all $j\geq 0$.  Since
$\mathcal C_1^{r,\varphi}$ is stable under extensions, the claim follows.
\end{proof}

\begin{lemma}\label{pelemcrit}
Let $M$ be a finitely generated $\mS$-module.  Suppose that
\[
  \gr_p^jM=p^jM/p^{j+1}M
\]
is finite free over $\mS_1=k[[z]]$ for every $j\geq 0$.  Then there exist
integers $m,n\geq 0$ and $a_i\geq 1$ such that
\[
  M\simeq
  \mS^{\oplus m}
  \oplus
  \bigoplus_{i=1}^n \mS/p^{a_i}.
\]
\end{lemma}

\begin{proof}
Let $T=M[p^\infty]$ be the $p$-power torsion submodule of $M$.  Since $M$ is
finitely generated over the noetherian ring $\mS$, there exists $N\gg 0$
such that $p^NT=0$.  Put $F=M/T$.

For $j\geq N$, multiplication by $p^j$ identifies
\[
  F/pF
  \simeq
  p^jM/p^{j+1}M
  =
  \gr_p^jM.
\]
Hence $F/pF$ is finite free over $\mS_1$.  Since $F$ is $p$-torsion-free,
the local criterion for flatness implies that $F$ is finite flat over
$\mS$.  As $\mS$ is local, $F$ is finite free over $\mS$.

The exact sequence
\[
  0\longrightarrow T\longrightarrow M\longrightarrow F\longrightarrow 0
\]
therefore splits as a sequence of $\mS$-modules.  Thus it remains to prove
the claim for the $p$-power torsion module $T$.  Since the sequence splits,
each $\gr_p^jT$ is a direct summand of $\gr_p^jM$, and hence is finite free
over $\mS_1$.

Assume now that $T$ is killed by $p^N$, and put
\[
  G_j:=\gr_p^jT
  \qquad
  0\leq j\leq N-1.
\]
Multiplication by $p$ induces surjective $\mS_1$-linear maps
\[
  \mu_j:G_j\longrightarrow G_{j+1}
  \qquad
  0\leq j<N-1.
\]
Since all $G_j$ are finite free over the principal ideal domain $k[[z]]$,
we may choose a basis of $G_0$ adapted to the successive surjections
$\mu_j$.  Equivalently, there are basis elements $e_1,\dots,e_t$ of $G_0$
and integers $a_i\geq 1$ such that, for every $j$, the nonzero elements
\[
  \mu_{j-1}\circ\cdots\circ\mu_0(e_i)
  \qquad
  \text{with } a_i>j
\]
form a basis of $G_j$.

Choose lifts $x_i\in T$ of $e_i$.  After replacing $x_i$ by $x_i-py_i$ for
suitable $y_i\in T$, we may assume
\[
  p^{a_i}x_i=0.
\]
Indeed, the condition that the chain attached to $e_i$ stops at length
$a_i$ says that $p^{a_i}x_i\in p^{a_i+1}T$, so such an adjustment kills
$p^{a_i}x_i$ without changing the class of $x_i$ in $G_0$.

Now define
\[
  \bigoplus_{i=1}^t \mS/p^{a_i}
  \longrightarrow
  T
\]
by sending the standard generator of the $i$-th summand to $x_i$.  By
construction, this map induces an isomorphism on every $p$-graded piece.
Since the $p$-adic filtration on $T$ is finite, the map itself is an
isomorphism.  Therefore
\[
  T\simeq \bigoplus_{i=1}^t \mS/p^{a_i}.
\]
Combining this with the finite free decomposition of $F$ gives the desired
decomposition of $M$.
\end{proof}

\begin{proposition}\label{modfieq}
Assume $er<p-1$.  Then
\[
  \overline{\ModFI}^{r,\varphi}_{/\mS_\infty}
  =
  \overline{\Mod}^{r,\varphi}_{/\mS_\infty}.
\]
In particular, every object $M\in\overline{\Mod}^{r,\varphi}_{/\mS_\infty}$
has underlying $\mS$-module of the form
\[
  M\simeq
  \mS^{\oplus m}
  \oplus
  \bigoplus_{i=1}^n \mS/p^{a_i}.
\]
\end{proposition}

\begin{proof}
The inclusion
\[
  \overline{\ModFI}^{r,\varphi}_{/\mS_\infty}
  \subset
  \overline{\Mod}^{r,\varphi}_{/\mS_\infty}
\]
is immediate from the definition.  We prove the reverse inclusion.

Let $M$ be an object of $\overline{\Mod}^{r,\varphi}_{/\mS_\infty}$.  By the
definition of this category as an extension closure, $M$ admits a finite
filtration
\[
  0=M_0\subset M_1\subset \cdots \subset M_s=M
\]
in $\widetilde{\Mod}^{\infty,\varphi}_{/\mS}$ such that each graded piece
\[
  H_i:=M_i/M_{i-1}
\]
lies in $\overline{\Mod}^{r,\varphi}_{/\mS_1}$.  Thus each $H_i$ is either
finite free over $\mS_1$ and lies in $\Mod^{r,\varphi}_{/\mS_1}$, or finite
free over $\mS$ and lies in $\widetilde{\Mod}^{r,\varphi}_{/\mS}$.

We prove by induction on $i$ that
\[
  \gr_p^j M_i\in\mathcal C_1^{r,\varphi}
  \qquad
  \text{for all } j\geq 0.
\]
For $i=0$, this is clear.  Suppose the assertion is known for $M_{i-1}$.

If $H_i$ is finite free over $\mS_1$, then $H_i$ lies in
$\Mod^{r,\varphi}_{/\mS_1}$.  Applying Lemma~\ref{grpextone} to
\[
  0\longrightarrow M_{i-1}
  \longrightarrow M_i
  \longrightarrow H_i
  \longrightarrow 0
\]
shows that $\gr_p^jM_i\in\mathcal C_1^{r,\varphi}$ for every $j\geq 0$.

If $H_i$ is finite free over $\mS$, then $H_i$ lies in
$\widetilde{\Mod}^{r,\varphi}_{/\mS}$.  Applying Lemma~\ref{grpextfree} to the same
exact sequence gives again
$\gr_p^jM_i\in\mathcal C_1^{r,\varphi}$ for every $j\geq 0$.

This proves the induction.  Therefore every $\gr_p^jM$ belongs to
$\mathcal C_1^{r,\varphi}$, and in particular is finite free over
$\mS_1=k[[z]]$.

By Lemma~\ref{pelemcrit}, the underlying $\mS$-module of $M$ is of the form
\[
  M\simeq
  \mS^{\oplus m}
  \oplus
  \bigoplus_{i=1}^n \mS/p^{a_i}.
\]
Hence $M$ belongs to
$\overline{\ModFI}^{r,\varphi}_{/\mS_\infty}$, proving the reverse
inclusion.
\end{proof}

\section{Proof of the main theorems}
Throughout this section, we retain the notation of Section~3. We regard $\sO_K$ as an $\bS [z]$-algebra via $z\mapsto \pi$. For a smooth scheme $X$ over $\sO_K$, let $\sX$ denote its $p$-adic formal completion. Set
\begin{align*}
X_{\sO_{\sC}}
&:= X \times_{\Spec(\sO_K)} \Spec(\sO_{\sC}), \\
\sX_{\sO_{\sC}}
&:= \sX \widehat{\times}_{\Spf(\sO_K)} \Spf(\sO_{\sC}).
\end{align*}
Then $\sX_{\sO_{\sC}}$ is canonically identified with the $p$-adic formal completion of $X_{\sO_{\sC}}$.

\begin{proposition}\label{tptcformal}
Assume that the invariants of the formal schemes below are defined via their quasisyntomic sites. Then there are canonical equivalences
\[
\TP(\sX_{\sO_{\sC}};\Z_p)
\simeq
\TP(X_{\sO_{\sC}};\Z_p),
\qquad
\TCn(\sX_{\sO_{\sC}};\Z_p)
\simeq
\TCn(X_{\sO_{\sC}};\Z_p),
\]
and, respectively,
\[
\TP(\sX/\bS[z];\Z_p)
\simeq
\TP(X/\bS[z];\Z_p),
\qquad
\TCn(\sX/\bS[z];\Z_p)
\simeq
\TCn(X/\bS[z];\Z_p).
\]
\end{proposition}

\begin{proof} We prove the assertion for $\sX_{\sO_{\sC}}$. The case of $\sX$ over $\bS[z]$ is entirely analogous. Choose an affine Zariski hypercover $U_\bullet \to X_{\sO_{\sC}}$ such that each
$U_n$ is a disjoint union of affine opens, and let $\mf{U}_\bullet$ be the termwise
$p$-adic formal completion of $U_\bullet$. Since $\TP(-;\Z_p)$ and
$\TCn(-;\Z_p)$ satisfy Zariski descent on schemes, while on the formal side
they are defined as global sections of the corresponding quasisyntomic sheaves, it is
enough to prove the claim for each affine component
$U = \Spec(A)$ occurring in some $U_n$, with formal completion
$\mf{U} = \Spf(\widehat A)$.

For such an affine $U$, by definition of the quasisyntomic site of an affine formal scheme,
we have
\[
\TP(\mf{U};\Z_p) \simeq \TP(\widehat A;\Z_p),
\qquad
\TCn(\mf{U};\Z_p) \simeq \TCn(\widehat A;\Z_p).
\]
On the other hand, applying \cite[Proposition 4.2]{AMNN22} first to
$\Spec(A)$ and then to $\Spec(\widehat A)$, yields equivalences
\[
\operatorname{THH}(A;\Z_p)
\xrightarrow{\sim}
\varprojlim_n \operatorname{THH}(A/p^nA;\Z_p)
\xleftarrow{\sim}
\operatorname{THH}(\widehat A;\Z_p),
\]
because $\widehat A/p^n\widehat A \cong A/p^nA$ for all $n \ge 1$. Hence
\[
\operatorname{THH}(A;\Z_p)
\simeq
\operatorname{THH}(\widehat A;\Z_p)
\]
as $S^1$-equivariant spectra.

Now
\[
\TCn(A;\Z_p)=\operatorname{THH}(A;\Z_p)^{hS^1},
\qquad
\TP(A;\Z_p)=\operatorname{THH}(A;\Z_p)^{tS^1},
\]
and similarly for $\widehat A$. Therefore the preceding equivalence implies
\[
\TCn(A;\Z_p) \simeq \TCn(\widehat A;\Z_p),
\qquad
\TP(A;\Z_p) \simeq \TP(\widehat A;\Z_p).
\]
Combining this with the identification of the formal side, we obtain
\[
\TCn(U;\Z_p) \simeq \TCn(\mf{U};\Z_p),
\qquad
\TP(U;\Z_p) \simeq \TP(\mf{U};\Z_p).
\]

These equivalences are functorial in $U$, hence assemble into equivalences of cosimplicial spectra
\[
\TCn(U_\bullet;\Z_p) \simeq \TCn(\mf{U}_\bullet;\Z_p),
\qquad
\TP(U_\bullet;\Z_p) \simeq \TP(\mf{U}_\bullet;\Z_p).
\]
Taking totalizations and using fpqc-descent~\cite[Corollary 3.4]{BMS19}, we get
\[
\TCn(X_{\sO_{\sC}};\Z_p)
\simeq
\operatorname{Tot}\bigl(\TCn(U_\bullet;\Z_p)\bigr)
\simeq
\operatorname{Tot}\bigl(\TCn(\mf{U}_\bullet;\Z_p)\bigr)
\simeq
\TCn(\sX_{\sO_{\sC}};\Z_p),
\]
and similarly
\[
  \TP(X_{\sO_{\sC}};\Z_p)
\simeq
\operatorname{Tot}\bigl(\TP(U_\bullet;\Z_p)\bigr)
\simeq
\operatorname{Tot}\bigl(\TP(\mf{U}_\bullet;\Z_p)\bigr)
\simeq
\TP(\sX_{\sO_{\sC}};\Z_p).
\]
Taking inverses of these equivalences gives the displayed maps
\[
\TP(\sX_{\sO_{\sC}};\Z_p)
\xrightarrow{\sim}
\TP(X_{\sO_{\sC}};\Z_p),
\]
and
\[
\TCn(\sX_{\sO_{\sC}};\Z_p)
\xrightarrow{\sim}
\TCn(X_{\sO_{\sC}};\Z_p).
\]
This proves the claim.
\end{proof}
The same argument gives the following result.
\begin{proposition}\label{hpformal}
Assume that $\HP(\sX/\sO_K;\Z_p)$ is defined via the quasisyntomic
site of the formal scheme $\sX$. Then there are canonical equivalences
\[
 \HP(\sX/\sO_K;\Z_p)
\xrightarrow{\simeq}
\HP(X/\sO_K;\Z_p).
\]
\end{proposition}

\begin{proof}
The proof is identical to that of Proposition~\ref{tptcformal}: one chooses an affine Zariski hypercover of $X$, passes to the termwise $p$-adic formal completions, uses $p$-adic continuity for Hochschild homology, and then takes Tate fixed points after passing to periodic cyclic homology. Descent identifies the resulting totalizations, giving the displayed equivalence.
\end{proof}

\subsection{Preliminaries on spectral sequences II}

In this subsection, we let $R$ be a commutative ring and discuss several preliminary facts concerning spectral sequences that will be used later.

Let $M$ be a spectrum equipped with a decreasing filtration
\begin{equation*}
  \cdots \longrightarrow \fil^{n+1}M
  \longrightarrow \fil^n M
  \longrightarrow \fil^{n-1}M
  \longrightarrow \cdots
  \longrightarrow M.
\end{equation*}
Set
\begin{equation*}
  \gr^n M :=
  \operatorname{cofib}\bigl(\fil^{n+1}M\to \fil^n M\bigr).
\end{equation*}
We assume throughout that the filtration is complete and exhaustive. We also assume that the induced spectral sequence
\begin{equation*}
  E_1^{n,*}=\pi_*(\gr^n M)\Longrightarrow \pi_*(M)
\end{equation*}
is strongly convergent and locally finite in each homotopy degree. More explicitly, for every $i$, only finitely many $n$ contribute to the filtration on $\pi_i(M)$. Finally, we assume that all groups $\pi_i(M)$, $\pi_i(\fil^n M)$, and $\pi_i(\gr^n M)$ are finitely generated $R$-modules, and that all maps appearing in the exact couple are $R$-linear. Let
\begin{equation*}
  F^n\pi_i(M):=
  \operatorname{im}\bigl(\pi_i(\fil^nM)\to \pi_i(M)\bigr)
\end{equation*}
be the induced filtration on $\pi_i(M)$, and set
\[
  \gr_F^n(\pi_i M):=F^n \pi_i M/F^{n+1} \pi_i M.
\]

\begin{defn}
Let $R \to A$ be a morphism of commutative rings.
\begin{itemize}
  \item[(1)] We say that the filtered spectrum $M$ is \emph{degenerate after tensoring with $A$} if for all $i,n\in\Z$, the natural map
  \[
    \pi_i(\fil^n M) \otimes_R A \to \pi_i(M)\otimes_R A
  \]
  is injective.

  \item[(2)] We say that the filtered spectrum $M$ is \emph{split degenerate after tensoring with $A$} if for all $i,n\in\Z$, the natural map
  \[
    \pi_i(\fil^n M) \otimes_R A \to \pi_i(M)\otimes_R A
  \]
  is split injective as a map of $A$-modules.
\end{itemize}
\end{defn}

The following proposition is proved by the same argument as Proposition~\ref{sscrit}(2).

\begin{proposition}\label{sscritflat}
Let $R \to A$ be a flat morphism of commutative rings such that $A$ is a DVR with fraction field $K$. Assume that the spectral sequence associated with $M$ degenerates after tensoring with $K$. Then the filtered spectrum $M$ is split degenerate after tensoring with $A$ if and only if, for every $i\in\Z$, there is an abstract isomorphism of $A$-modules
\begin{equation*}
  \bigl( \pi_i(M)\otimes_R A \bigr)_{\tors}
  \simeq
  \bigoplus_n \bigl(\pi_i(\gr^nM)\otimes_R A \bigr)_{\tors}.
\end{equation*}
\end{proposition}

\begin{proof}
The proof is identical to that of Proposition~\ref{sscrit}(2), applied to the filtered $A$-module
\begin{equation*}
  N_i:=\pi_i(M)\otimes_R A
\end{equation*}
with filtration
\begin{equation*}
  F_A^nN_i:=
  \operatorname{im}\bigl(\pi_i(\fil^n M)\otimes_R A
  \to \pi_i(M)\otimes_R A\bigr).
\end{equation*}
Since the associated spectral sequence degenerates after tensoring with $K$, the same length argument as in Proposition~\ref{sscrit}(2) shows that the filtration $F_A^* N_i$ is split if and only if
\begin{equation*}
  (N_i)_{\tors}
  \simeq
  \bigoplus_n \bigl(\gr_{F_A}^n N_i\bigr)_{\tors}.
\end{equation*}
By construction, one has
\begin{equation*}
  \gr_{F_A}^n N_i
  \cong
  \operatorname{im}\bigl(\pi_i(\fil^n M)\otimes_R A
  \to \pi_i(M)\otimes_R A\bigr)
  \Big/
  \operatorname{im}\bigl(\pi_i(\fil^{n+1} M)\otimes_R A
  \to \pi_i(M)\otimes_R A\bigr),
\end{equation*}
so, under the evident identification of the graded pieces with
$\pi_i(\gr^n M)\otimes_R A$, the claim follows.
\end{proof}

We also prove the following lemma.

\begin{lemma}\label{degdetect}
Let $M$ be a filtered spectrum as above, and let $R \to A$ be a morphism of commutative rings such that $A$ is a DVR with fraction field $K$. Assume that for every $n,i\in \Z$, the natural map
\begin{equation*}
  \pi_i(\gr^n M)\longrightarrow \pi_i(\gr^n M)\otimes_R A
\end{equation*}
is injective, and that $M$ is degenerate after tensoring with $A$. Then the original spectral sequence is degenerate.
\end{lemma}

\begin{proof}
Let $E_r^{a,b}$ be the associated spectral sequence. Since
\begin{equation*}
  E_1^{a,b}=\pi_{a+b}(\gr^{-a}M),
\end{equation*}
the map
\begin{equation*}
  E_1^{a,b}\to E_1^{a,b}\otimes_R A
\end{equation*}
is injective for every $(a,b)$ by assumption.

We prove by induction on $r\geq 1$ that $d_r=0$. Assume that $d_s=0$ for all $s<r$. Then there are canonical identifications
\begin{equation*}
  E_r^{a,b}=E_1^{a,b}
\end{equation*}
for all $(a,b)$. Since $M$ is degenerate after tensoring with $A$, the spectral sequence after tensoring with $A$ is degenerate, and hence
\begin{equation*}
  d_r\otimes_R A=0.
\end{equation*}
Therefore $d_r=0$, because its target
\begin{equation*}
  E_r^{a+r,b-r+1}=E_1^{a+r,b-r+1}
\end{equation*}
injects into
\begin{equation*}
  E_1^{a+r,b-r+1}\otimes_R A.
\end{equation*}
Thus all differentials vanish, and the spectral sequence is degenerate.
\end{proof}
\subsection{\texorpdfstring{Degeneration of the Breuil--Kisin-to-$\TCn$ spectral sequence}{Degeneration of the Breuil--Kisin-to-TC spectral sequence}}

We recall that
\[
  \pi_* \TCn(\sO_K/\bS[z];\Z_p)\simeq \mS[u,v]/(uv-E),
\]
where $\deg u=2$ and $\deg v=-2$, and that
\[
  \pi_* \TP(\sO_K/\bS[z];\Z_p)\simeq \mS[\sigma^{\pm 1}],
\]
where $\deg \sigma=2$. Moreover, $\can_{\sO_K}$ is $\sO_K$-linear and sends $u$ to $E\sigma$ and $v$ to $\sigma^{-1}$, while $\varphi_{\sO_K}^{h\mathbb{T}}$ is $\varphi$-semilinear and sends $u$ to $\sigma$ and $v$ to $\varphi(E)\sigma^{-1}$; see \cite[Proposition 11.10]{BMS19}.

Let $X$ be a smooth variety over $\sO_K$. We denote by the two filtered objects below the global sections of the quasisyntomic sheafifications of the double-speed Postnikov filtrations on $\TCn(-;\Z_p)$ and $\TP(-;\Z_p)$:
\begin{align*}
\fil_\BMS^*\TCn(\sX/\bS[z];\Z_p)
&:= R\Gamma_{\syn}\bigl(\sX,\fil_\BMS^* \TCn(-/\bS[z];\Z_p)\bigr),\\
\fil_\BMS^*\TP(\sX/\bS[z];\Z_p)
&:= R\Gamma_{\syn}\bigl(\sX,\fil_\BMS^* \TP(-/\bS[z];\Z_p)\bigr).
\end{align*}
These define complete, exhaustive, decreasing, multiplicative $\Z$-indexed filtrations; see \cite[Theorem 1.12]{BMS19}. We define
\begin{align*}
  H^m_{\mS}(\sX)\{n\}
  &:= \pi_{2n-m} R\Gamma_{\syn}\bigl(\sX,\gr_\BMS^{n}\TCn(-/\bS[z];\Z_p)[1/u]\bigr)\\
  &\simeq \pi_{2n-m}\gr_\BMS^{n}\TCn(\sX/\bS[z];\Z_p)[1/u],
\end{align*}
where $u\in \pi_* \TCn(\sO_K/\bS[z];\Z_p)$. We thus obtain a spectral sequence
\begin{equation}\label{bkss}
  E_2^{i,j}=H_{\mS}^{i-j}(\sX)\{j\} \Longrightarrow \pi_{-i-j}\bigl(\TCn(\sX/\bS[z];\Z_p)[1/u]\bigr).
\end{equation}

For a quasisyntomic ring $A$ over $\sO_K$, the cyclotomic Frobenius
\[
  \varphi_A^{h\mathbb{T}}:\TCn(A/\bS[z];\Z_p)\to \TP(A/\bS[z];\Z_p)
\]
sends $u$ to an invertible element. Hence it induces a map
\[
  \TCn(A/\bS[z];\Z_p)[1/u]\overset{\varphi_A^{h\mathbb{T}}}{\longrightarrow}\TP(A/\bS[z];\Z_p).
\]
Passing to filtered objects, we obtain a morphism of filtered spectra
\[
  \fil_{\BMS}^*\TCn(\sX/\bS[z];\Z_p)[1/u]
  \overset{\fil^*\varphi_{\sX}^{h\mathbb{T}}}{\longrightarrow}
  \fil_{\BMS}^*\TP(\sX/\bS[z];\Z_p),
\]
and hence a morphism on associated graded pieces
\[
  \gr_{\BMS}^*\TCn(\sX/\bS[z];\Z_p)[1/u]
  \overset{\gr^*\varphi_{\sX}^{h\mathbb{T}}}{\longrightarrow}
  \gr_{\BMS}^*\TP(\sX/\bS[z];\Z_p).
\]

On the other hand, the canonical map
\[
  \can_A:\TCn(A/\bS[z];\Z_p)\to \TP(A/\bS[z];\Z_p)
\]
also induces a morphism of filtered spectra
\begin{equation}\label{filcan}
  \fil_{\BMS}^*\TCn(\sX/\bS[z];\Z_p)
  \overset{\fil^*\can_{\sX}}{\longrightarrow}
  \fil_{\BMS}^*\TP(\sX/\bS[z];\Z_p),
\end{equation}
and hence a morphism on associated graded pieces
\begin{equation}\label{grcan}
  \gr_{\BMS}^*\TCn(\sX/\bS[z];\Z_p)
  \overset{\gr^*\can_{\sX}}{\longrightarrow}
  \gr_{\BMS}^*\TP(\sX/\bS[z];\Z_p).
\end{equation}

\begin{lemma}
After inverting $E$, the morphisms \eqref{filcan} and \eqref{grcan} are equivalences.
\end{lemma}

\begin{proof}
By \cite[Proposition 4.31]{BMS19}, it suffices to show that for every quasiregular semiperfectoid ring $S$ over $\sO_K$, the morphisms
\begin{equation}\label{filcanaff}
  \fil_{\BMS}^*\TCn(S/\bS[z];\Z_p)
  \overset{\fil^*\can_S}{\longrightarrow}
  \fil_{\BMS}^*\TP(S/\bS[z];\Z_p)
\end{equation}
and
\begin{equation}\label{grcanaff}
  \gr_{\BMS}^*\TCn(S/\bS[z];\Z_p)
  \overset{\gr^*\can_S}{\longrightarrow}
  \gr_{\BMS}^*\TP(S/\bS[z];\Z_p)
\end{equation}
become equivalences after inverting $E$.

Since $\TCn(S/\bS[z];\Z_p)$ and $\TP(S/\bS[z];\Z_p)$ are concentrated in even degrees, it suffices to show that
\[
  \TCn_{2n}(S/\bS[z];\Z_p)[1/E]
  \overset{\pi_{2n}\can_S}{\longrightarrow}
  \TP_{2n}(S/\bS[z];\Z_p)[1/E]
\]
is an isomorphism for every $n$. If $n\leq 0$, this follows from \cite[Proposition 11.11]{BMS19}. If $n>0$, consider multiplication by $v^n$:
\[
  \TCn_{2n}(S/\bS[z];\Z_p)[1/E]\longrightarrow \TCn_0(S/\bS[z];\Z_p)[1/E].
\]
Since $uv=E$, the element $v$ becomes invertible after inverting $E$, so this map is an isomorphism. The claim therefore reduces to the case $n=0$, which is already known.
\end{proof}

The preceding lemma yields morphisms
\begin{align}
  \fil_{\BMS}^*\TCn(\sX/\bS[z];\Z_p)[1/u]
  &\overset{\fil^*\varphi_{\sX}^{h\mathbb{T}}}{\longrightarrow}
  \fil_{\BMS}^*\TP(\sX/\bS[z];\Z_p) \notag\\
  &\longrightarrow
  \fil_{\BMS}^*\TP(\sX/\bS[z];\Z_p)[1/E] \notag\\
  &\overset{\fil^*\can_{\sX}}{\simeq}
  \fil_{\BMS}^*\TCn(\sX/\bS[z];\Z_p)[1/E] \notag\\
  &=
  \fil_{\BMS}^*\TCn(\sX/\bS[z];\Z_p)[1/uE]
  \label{filfrobcan}
\end{align}
and
\begin{align}
  \gr_{\BMS}^*\TCn(\sX/\bS[z];\Z_p)[1/u]
  &\overset{\gr^*\varphi_{\sX}^{h\mathbb{T}}}{\longrightarrow}
  \gr_{\BMS}^*\TP(\sX/\bS[z];\Z_p) \notag\\
  &\longrightarrow
  \gr_{\BMS}^*\TP(\sX/\bS[z];\Z_p)[1/E] \notag\\
  &\overset{\gr^*\can_{\sX}}{\simeq}
  \gr_{\BMS}^*\TCn(\sX/\bS[z];\Z_p)[1/E] \notag\\
  &=
  \gr_{\BMS}^*\TCn(\sX/\bS[z];\Z_p)[1/uE]
  \label{grfrobcan}
\end{align}
where we use $uv=E$. On homotopy groups, these morphisms are $\varphi$-semilinear.

\begin{prop}\label{bkgr}
The $\mS$-linearization of the homotopy groups of \eqref{grfrobcan},
\begin{equation}\label{bkgrmor}
  \pi_m \gr_{\BMS}^n\TCn(\sX/\bS[z];\Z_p)[1/u]
  \otimes_{\mS,\varphi}\mS[1/E]
  \longrightarrow
  \pi_m \gr_{\BMS}^n\TCn(\sX/\bS[z];\Z_p)[1/u][1/E],
\end{equation}
is an isomorphism for all $m,n\in\Z$. Moreover, if $X$ is smooth proper over $\sO_K$, then
\[
  H^m_{\mS}(\sX)\{n\}:=
  \pi_{2n-m}\gr_{\BMS}^n\TCn(\sX/\bS[z];\Z_p)[1/u]
\]
admits the structure of a Breuil--Kisin module of $E$-height in $[-n,m-n]$ for $0\leq m\leq 2\dim(X/\sO_K)$, and vanishes for $m<0$ or $m>2\dim(X/\sO_K)$.
\end{prop}

\begin{proof}
In the case $n=0$, this follows from \cite[Proposition 11.15]{BMS19} and \cite[Theorem 1.8]{BS22}. For general $n$, the claim follows from the commutative diagram
\tiny
\[
\xymatrix{
\pi_{2n-m}\gr_{\BMS}^n\TCn(\sX/\bS[z];\Z_p)[1/u]
  \ar[d]_-{\simeq}^{u^{-n}\cdot}
  \ar[r]^-{\varphi_{\sX}^{h\mathbb{T}}}
&
\pi_{2n-m}\gr_{\BMS}^n\TCn(\sX/\bS[z];\Z_p)[1/E]
  \ar[d]^-{\simeq}_{\sigma^{-n}\cdot}
&
\pi_{2n-m}\gr_{\BMS}^n\TCn(\sX/\bS[z];\Z_p)[1/uE]
  \ar[l]^-{\can_{\sX}}_-{\simeq}
  \ar[d]^-{\simeq}_{E^nu^{-n}\cdot}
\\
\pi_{-m}\gr_{\BMS}^{0}\TCn(\sX/\bS[z];\Z_p)[1/u]
  \ar[r]^-{\varphi_{\sX}^{h\mathbb{T}}}
&
\pi_{-m}\gr_{\BMS}^{0}\TCn(\sX/\bS[z];\Z_p)[1/E]
&
\pi_{-m}\gr_{\BMS}^{0}\TCn(\sX/\bS[z];\Z_p)[1/uE]
  \ar[l]^-{\can_{\sX}}_-{\simeq}
}
\]\normalsize
and the case $n=0$.
\end{proof}

\begin{prop}\label{bkdegCflat}
Let $X$ be a smooth proper variety over $\sO_K$. Assume that $\dim X<p$. Then the spectral sequence
\[
  E_2^{i,j}=H_{\mS}^{i-j}(\sX)\{j\} \Longrightarrow \pi_{-i-j}\bigl(\TCn(\sX/\bS[z];\Z_p)[1/u]\bigr)
\]
is split degenerate after tensoring with $W(\sC^\flat)$.
\end{prop}

\begin{proof}
We note that, as $W(\sC^\flat)$-modules,
\[
  H_{\mS}^{i-j}(\sX)\{j\}\otimes_{\mS}W(\sC^\flat)
  \simeq
  H_{\mS}^{i-j}(\sX)\otimes_{\mS}W(\sC^\flat).
\]
By Proposition~\ref{sscritflat}, it suffices to show the following isomorphism of $W(\sC^\flat)$-modules:
\begin{equation}\label{Cflatisom}
  \TCn_n(\sX/\bS[z];\Z_p)[1/u]\otimes_{\mS}W(\sC^\flat)
  \simeq
  \bigoplus_{i\in \Z} H_{\mS}^{-n+2i}(\sX)\otimes_{\mS}W(\sC^\flat).
\end{equation}
Since $\TCn_n(\sX/\bS[z];\Z_p)[1/u]$ is $2$-periodic, we may assume that $n$ is sufficiently large. Then we have
\[
  \TCn_n(\sX/\bS[z];\Z_p)[1/u] \simeq \TCn_n(\sX/\bS[z];\Z_p)
  \]
(see \cite[Proposition 2.6]{M26}).

By Proposition~\ref{tptcformal}, there is a canonical isomorphism
\[
  \TCn_n(\sX/\bS[z];\Z_p)\simeq \TCn_n(X/\bS[z];\Z_p).
\]
By Theorem~\ref{ncbms}, its \'etale realization is identified with $\LK K_n(X_\sC)$. On the other hand, the \'etale realization of $H^i_{\mS}(\sX)$ is $H^i_{\et}(\mf{X}_\sC,\Z_p)$ by \cite[Theorem 1.2]{BMS19}, and this is canonically isomorphic to $H^i_{\et}(X_\sC,\Z_p)$ by \cite{F95}. Hence Theorem~\ref{ktopchdec} yields an isomorphism of $\Z_p$-modules
\[
  \LK K_n(X_\sC)\simeq \bigoplus_{i\in \Z} H_{\et}^{-n+2i}(X_\sC,\Z_p).
\]
By the equivalence between the category of $W(\sC^\flat)$-modules equipped with a Frobenius automorphism and the category of finitely generated $\Z_p$-modules; see \cite[Theorem 3.2.5]{BC09}, we obtain \eqref{Cflatisom}.
\end{proof}

\begin{prop}\label{bkdeg}
Let $X$ be a smooth proper variety over $\sO_K$. Assume that
\[
  e\cdot (2\dim(X/\sO_K)-1)< p-1.
\]
Then the spectral sequence
\[
  E_2^{i,j}=H_{\mS}^{i-j}(\sX)\{j\} \Longrightarrow \pi_{-i-j}\bigl(\TCn(\sX/\bS[z];\Z_p)[1/u]\bigr)
\]
is degenerate.
\end{prop}

\begin{proof}
By \cite[Theorem 5.11]{M21}, one has
\[
  H_{\mS}^{*}(\sX)\simeq H_{\et}^{*}(X_\sC,\Z_p)\otimes_{\Z_p}\mS.
\]
Thus the natural map
\[
  H_{\mS}^{*}(\sX)\to H_{\mS}^{*}(\sX)\otimes_{\mS}W(\sC^\flat)
\]
is injective. The claim now follows from Lemma~\ref{degdetect} and Proposition~\ref{bkdegCflat}.
\end{proof}


\subsection{\texorpdfstring{Split degeneration of the Breuil--Kisin-to-$\TCn$ spectral sequence}{Split degeneration of the Breuil--Kisin-to-TC spectral sequence}}

\begin{prop}
Let $X$ be a smooth proper variety over $\sO_K$. Assume that
\[
  e\cdot (2\dim(X/\sO_K)-1)< p-1.
\]
Then the $\mS$-linearization of the homotopy groups of \eqref{filfrobcan},
\begin{equation}\label{bkfilmor}
  \pi_m \fil_{\BMS}^n\TCn(\sX/\bS[z];\Z_p)[1/u]
  \otimes_{\mS,\varphi}\mS[1/E]
  \longrightarrow
  \pi_m \fil_{\BMS}^n\TCn(\sX/\bS[z];\Z_p)[1/u][1/E],
\end{equation}
is an isomorphism for all $m,n\in\Z$. In particular, if $X$ is smooth proper over $\sO_K$, then
\[
  \pi_m \fil_{\BMS}^n\TCn(\sX/\bS[z];\Z_p)[1/u]
\]
admits the structure of a Breuil--Kisin module.
\end{prop}

\begin{proof}
By Proposition~\ref{bkdeg}, there is a short exact sequence
\begin{equation}\label{bkshort}
  0 \to \pi_m \fil_{\BMS}^{n+1}\TCn(\sX/\bS[z];\Z_p)[1/u]
  \to \pi_m \fil_{\BMS}^{n}\TCn(\sX/\bS[z];\Z_p)[1/u]
  \to H_{\mS}^{2n-m}(\sX)\{n\}
  \to 0.
\end{equation}
Since the filtration $\fil_{\BMS}^*\TCn(\sX/\bS[z];\Z_p)[1/u]$ is exhaustive, we have
\[
  \lim_n \pi_m \fil_{\BMS}^{n}\TCn(\sX/\bS[z];\Z_p)[1/u]=0.
\]
Moreover, since $H_{\mS}^{2n-m}(\sX)\{n\}=0$ for $2n-m>2\dim(X/\sO_K)$, we obtain
\[
  \pi_m \fil_{\BMS}^{n}\TCn(\sX/\bS[z];\Z_p)[1/u]
  \simeq
  \pi_m \fil_{\BMS}^{n+1}\TCn(\sX/\bS[z];\Z_p)[1/u]
  = \cdots =0
\]
for $n>m/2+\dim(X/\sO_K)$. Fix $m\in\Z$. The claim now follows by descending induction on $n$, using \eqref{bkshort} and Proposition~\ref{bkgr}.
\end{proof}
\begin{lemma}\label{torssplit}
 Let
\[
  0\to A \xrightarrow{\alpha} B \xrightarrow{\beta} C\to 0
\]
be a short exact sequence of Breuil--Kisin modules over $\mS$ whose underlying $\mS$-modules are of the form
\[
  A\simeq \mS^{\oplus a}\oplus  \bigoplus_i \mS/(p^{a_i}),\qquad
  B\simeq \mS^{\oplus b}\oplus \bigoplus_i \mS/(p^{b_i}),\qquad
  C\simeq \mS^{\oplus c}\oplus \bigoplus_i \mS/(p^{c_i}).
\]
If, after tensoring with $W(\sC^\flat)$, the
induced sequence
\[
  0\to A_{W(\sC^\flat)}\to B_{W(\sC^\flat)}\to C_{W(\sC^\flat)}\to 0
\]
is split exact, where $M_W:=M\otimes_{\mS}W$, then
\[
  0\to A_{\tors}\to B_{\tors}\to C_{\tors}\to 0
\]
is split exact as a sequence of $\mS$-modules.
\end{lemma}

\begin{proof}
We write
\[
  \ol{W}:=W(\sC^\flat).
\]
The map $\mS\to \ol{W}$ sends $z$ to a unit.  In particular, for every $n\geq 1$
the natural map
\[
  \mS/p^n\longrightarrow \ol{W}/p^n
\]
is injective.  We shall use this repeatedly.

First we record an elementary injectivity statement for extension classes.
Let
\[
  T=\bigoplus_i \mS/p^{a_i},
  \qquad
  T'=\bigoplus_j \mS/p^{b_j},
\]
Then the natural map
\[
  \Ext^1_{\mS}(T,T')
  \longrightarrow
  \Ext^1_{\ol{W}}(T_{\ol{W}},T'_{\ol{W}})
\]
is injective.  Indeed, by additivity it is enough to treat the case
$T=\mS/p^a$ and $T'=\mS/p^b$.  The free resolution
\[
  0\longrightarrow \mS
  \xrightarrow{p^a}
  \mS
  \longrightarrow
  \mS/p^a
  \longrightarrow 0
\]
gives
\[
  \Ext^1_{\mS}(\mS/p^a,\mS/p^b)
  \simeq
  \mS/p^{\min(a,b)}.
\]
Similarly,
\[
  \Ext^1_{\ol{W}}({\ol{W}}/p^a,{\ol{W}}/p^b)
  \simeq
  {\ol{W}}/p^{\min(a,b)}.
\]
Under these identifications the base change map is the natural injection
\[
  \mS/p^{\min(a,b)}
  \longrightarrow
  {\ol{W}}/p^{\min(a,b)}.
\]
Thus the desired map on $\Ext^1$ is injective.

We next prove that the sequence on $p$-power torsion submodules is exact.
The injectivity of
\[
  A_{\tors}\longrightarrow B_{\tors}
\]
and the equality
\[
  \Ker(B_{\tors}\to C_{\tors})=A_{\tors}
\]
are immediate from the exactness of
\[
  0\to A\to B\to C\to 0.
\]
It remains to prove that
\[
  B_{\tors}\longrightarrow C_{\tors}
\]
is surjective.

Let $D:=\beta^{-1}(C_{\tors})\subset B$.  Then we have an exact sequence
\[
  0\to A\to D\to C_{\tors}\to 0.
\]
Quotienting by $A_{\tors}$ gives
\[
  0\to A/A_{\tors}
  \longrightarrow
  D/A_{\tors}
  \longrightarrow
  C_{\tors}
  \to 0.
\]
Since $A/A_{\tors}$ is finite free over $\mS$ and $C_{\tors}$ is a direct
sum of modules $\mS/p^c$, the extension class lies in
\[
  \Ext^1_{\mS}(C_{\tors},A/A_{\tors}).
\]
After tensoring with ${\ol{W}}$, the original sequence
\[
  0\to A_{\ol{W}}\to B_{\ol{W}}\to C_{\ol{W}}\to 0
\]
is split exact by assumption.  Pulling this split sequence back along
\[
  (C_{\tors})_{\ol{W}}\longrightarrow C_{\ol{W}}
\]
shows that the base change of
\[
  0\to A/A_{\tors}
  \longrightarrow
  D/A_{\tors}
  \longrightarrow
  C_{\tors}
  \to 0
\]
to ${\ol{W}}$ is split exact.  By the injectivity of the base change map on
$\Ext^1$, the original extension is already split over $\mS$.

Thus there is an $\mS$-linear section
\[
  s:C_{\tors}\longrightarrow D/A_{\tors}.
\]
Choose $c\in C_{\tors}$ and let $\tilde b\in D$ be a lift of $s(c)$. Choose $N\geq 1$ such that $p^N c=0$. Since $s$ is $\mS$-linear, we have
\[
  p^N s(c)=s(p^N c)=0.
\]
Hence, if $\tilde b\in D$ is any lift of $s(c)$, then
\[
  p^N\tilde b\in A_{\tors}.
\]
Since $A_{\tors}$ is $p$-power torsion, there exists $N'\geq 1$ such that
\[
  p^{N'}(p^N\tilde b)=0.
\]
Thus $p^{N+N'}\tilde b=0$, so $\tilde b$ is $p$-power torsion. Viewing
$D$ as a submodule of $B$, we obtain $\tilde b\in B_{\tors}$. Moreover, the image of $\tilde b$ in $C_{\tors}$ is $c$. Therefore every
element of $C_{\tors}$ lifts to an element of $B_{\tors}$, and hence
$B_{\tors}\to C_{\tors}$ is surjective.

We have therefore obtained a short exact sequence
\[
  0\to A_{\tors}
  \longrightarrow
  B_{\tors}
  \longrightarrow
  C_{\tors}
  \to 0.
\]

It remains to prove that this short exact sequence splits.  Let
\[
  \xi\in \Ext^1_{\mS}(C_{\tors},A_{\tors})
\]
be its extension class.  After tensoring with $W$, the original sequence
splits:
\[
  B_{\ol{W}}\simeq A_{\ol{W}}\oplus C_{\ol{W}}.
\]
Since
\[
  A_{\tors},\quad B_{\tors},\quad C_{\tors}
\]
are direct sums of modules of the form $\mS/p^n$, their base changes are
exactly the $p$-power torsion submodules of $A_W$, $B_{\ol{W}}$, and $C_{\ol{W}}$.
Hence the induced sequence
\[
  0\to (A_{\tors})_{\ol{W}}
  \longrightarrow
  (B_{\tors})_{\ol{W}}
  \longrightarrow
  (C_{\tors})_{\ol{W}}
  \to 0
\]
is split exact.  Equivalently, the image of $\xi$ under
\[
  \Ext^1_{\mS}(C_{\tors},A_{\tors})
  \longrightarrow
  \Ext^1_{\ol{W}}((C_{\tors})_{\ol{W}},(A_{\tors})_{\ol{W}})
\]
is zero.

By the injectivity proved at the beginning of the proof, we have $\xi=0$.
Therefore
\[
  0\to A_{\tors}
  \longrightarrow
  B_{\tors}
  \longrightarrow
  C_{\tors}
  \to 0
\]
is split exact as a sequence of $\mS$-modules.
\end{proof}

\begin{lemma}\label{glsplit}
Let
\[
  0\to A \xrightarrow{\alpha} B \xrightarrow{\beta} C\to 0
\]
be a short exact sequence of $\mS$-modules. Assume that
\[
  0\to A_{\tors}\to B_{\tors}\to C_{\tors}\to 0
\]
is split exact. Assume moreover that the quotient
\[
  C_{\tf}:=C/C_{\tors}
\]
is finite free over $\mS$. Then $\beta:B\to C$ is split surjective.
\end{lemma}

\begin{proof}
Choose a splitting
\[
  s_{\tors}:C_{\tors}\to B_{\tors}\subset B
\]
of $B_{\tors}\to C_{\tors}$. Since $C\to C_{\tf}$ is split, we may write
\[
  C\simeq C_{\tors}\oplus C_{\tf}
\]
as an $\mS$-module. Let
\[
  \iota_{\tf}:C_{\tf}\to C
\]
be the corresponding section.

Since $C_{\tf}$ is projective and $\beta:B\to C$ is surjective, the map
$\iota_{\tf}$ lifts to an $\mS$-linear map
\[
  \widetilde{s}_{\tf}:C_{\tf}\to B
\]
such that
\[
  \beta\circ \widetilde{s}_{\tf}=\iota_{\tf}.
\]
Now define
\[
  s:C=C_{\tors}\oplus C_{\tf}\longrightarrow B
\]
by
\[
  s(x,y):=s_{\tors}(x)+\widetilde{s}_{\tf}(y).
\]
Then
\[
  \beta\circ s=\mathrm{id}_C.
\]
Therefore $\beta$ is split surjective.
\end{proof}

\begin{thm}\label{bksplit}
Let $X$ be a smooth proper variety over $\sO_K$, and put $d:=\dim(X/\sO_K)$. Assume $2de<p-1$. Then the spectral sequence
\[
  E_2^{i,j}
  =
  H_{\mS}^{i-j}(\sX)\{j\}
  \Longrightarrow
  \pi_{-i-j}\bigl(
    \TCn(\sX/\bS[z];\Z_p)[1/u]
  \bigr)
\]
is split degenerate. In particular, using \cite[Theorem 5.11]{M21}, there is an abstract
isomorphism
\[
  \pi_m\fil_{\BMS}^n
  \TCn(\sX/\bS[z];\Z_p)[1/u]
  \simeq
  \bigoplus_{i\geq 0}
  H_{\et}^{2(n+i)-m}(X_{\sC},\Z_p)
  \otimes_{\Z_p}\mS .
\]
\end{thm}

\begin{proof}
Write
\[
  \fM_{m,n}:=
  \pi_m\fil_{\BMS}^n
  \TCn(\sX/\bS[z];\Z_p)[1/u].
\]
By Proposition~\ref{bkdeg}, for every $m,n\in\Z$ we have a short exact
sequence
\begin{equation}\label{filshort}
  0
  \to
  \fM_{m,n+1}
  \to
  \fM_{m,n}
  \to
  H_{\mS}^{2n-m}(\sX)\{n\}
  \to
  0.
\end{equation}
Put
\[
  H_{m,n}:=H_{\mS}^{2n-m}(\sX)\{n\}.
\]
Then \eqref{filshort} takes the form
\begin{equation}\label{Mrow}
  0
  \to
  \fM_{m,n+1}
  \to
  \fM_{m,n}
  \to
  H_{m,n}
  \to
  0.
\end{equation}
It is enough to prove that \eqref{filshort} is split for every
$m,n$, since the filtration is finite in each total degree.

We first discuss the structure of the modules $\fM_{m,n}$.  The spectrum
\[
  \TCn(\sX/\bS[z];\Z_p)[1/u]
\]
is $2$-periodic, and hence there are isomorphisms of $\mS$-module
\begin{equation}\label{period}
  \fM_{m,n}\simeq \fM_{m+2,n+1}.
\end{equation}
Thus, according to the parity of $m$, it is enough to analyze
\[
  \fM_{-2d,n}
  \quad\text{and}\quad
  \fM_{-2d-1,n}.
\]

For $m=-2d$, the graded pieces in \eqref{filshort} are
$H_{\mS}^{2n+2d}(\sX)\{n\}$, and they vanish unless
$0\leq 2n+2d\leq 2d$. Moreover $H_{\mS}^{2n+2d}(\sX)\{n\}$ has $E$-height contained in $[0,2d]$; see
\cite[Theorem~1.8(6)]{BS21}. Thus, $H_{\mS}^{2n+2d}(\sX)\{n\}_{\tors}$ is in $\BKModinf^{2d,\varphi}$. By definition, $\BKModinf^{2d,\varphi}$ is a full subcategory of $\overline{\Mod}^{2d,\varphi}_{/\mS_\infty}$, so $H_{\mS}^{2n+2d}(\sX)\{n\}_{\tors}$ is in $\overline{\Mod}^{2d,\varphi}_{/\mS_\infty}$. By \cite[Theorem 5.11]{M21}, there is an isomorphism in $\BKMod^{2d,\varphi}$
\[
  H_{\mS}^{2n+2d}(\sX)\{n\} \simeq H_{\mS}^{2n+2d}(\sX)\{n\}_{\tors}\oplus H_{\mS}^{2n+2d}(\sX)\{n\}_{\free}.
  \] 
  Since the underlying $\mS$-module of $H_{\mS}^{2n+2d}(\sX)\{n\}_{\free}$ is finite free, it is in $\overline{\Mod}^{2d,\varphi}_{/\mS_1} \subset \overline{\Mod}^{2d,\varphi}_{/\mS_\infty}$. Since $\overline{\Mod}^{2d,\varphi}_{/\mS_\infty}$ is closed under extension, we obtain that 
  \[
    H_{\mS}^{2n+2d}(\sX)\{n\}
    \]
    is in $\overline{\Mod}^{2d,\varphi}_{/\mS_\infty}$. Using the exact sequence \eqref{filshort} and induction $n$, we get that $\fM_{-2d,n}$ is in $\overline{\Mod}^{2d,\varphi}_{/\mS_\infty}$. Since
\[
 2de<p-1,
\]
Proposition~\ref{modfieq} applies. Hence, for every $n$,
\begin{equation}\label{Mtors1}
  \fM_{-2d,n}
  \simeq \mS^{\oplus a} \oplus 
  \bigoplus_\lambda \mS/(p^{a_\lambda})
\end{equation}
for some finite multiset of positive integers $\{a_\lambda\}$ and a natural number $a$.

For $m=-2d-1$, by the same argument in the case $m=-2d$, for every $n$,
\begin{equation}\label{Mtors2}
  \fM_{-2d-1,n}
  \simeq\mS^{\oplus b} \oplus 
  \bigoplus_{\lambda'} \mS/(p^{b_{\lambda'}})
\end{equation}
for some finite multiset of positive integers $\{b_{\lambda'}\}$ and a natural number $b$.

By the isomorphism~\eqref{period}, we obtain that for every $m,n$,
\begin{equation}\label{Mtors}
  \fM_{m,n} 
  \simeq \mS^{\oplus a} \oplus 
  \bigoplus_\lambda \mS/(p^{a_\lambda})
\end{equation}
for some finite multiset of positive integers $\{a_\lambda\}$ and a natural number $a$.

We now prove that \eqref{filshort} splits. Put
\[
  H_{m,n}:=H_{\mS}^{2n-m}(\sX)\{n\}.
\]
Taking $p$-power torsion in \eqref{filshort} gives a left exact
sequence
\[
  0
  \to
  \fM_{m,n+1,\tors}
  \to
  \fM_{m,n,\tors}
  \to
  H_{m,n,\tors}.
\]
By Proposition~\ref{bkdegCflat}, after tensoring with $W(\sC^\flat)$ the
spectral sequence is split degenerate. In particular, the base change of
\eqref{filshort} to $W(\sC^\flat)$ is split exact. Therefore
Lemma~\ref{torssplit} gives a split exact sequence
\begin{equation}\label{torsrow}
  0
  \to
  \fM_{m,n+1,\tors}
  \to
  \fM_{m,n,\tors}
  \to
  H_{m,n,\tors}
  \to
  0.
\end{equation}

Next, by \cite[Theorem 5.11]{M21}, we have an abstract isomorphism
\[
  H_{\mS}^{q}(\sX)
  \simeq
  H_{\et}^{q}(X_{\sC},\Z_p)\otimes_{\Z_p}\mS.
\]
Since every finitely generated $\Z_p$-module is a direct sum of a free
$\Z_p$-module and cyclic modules $\Z_p/(p^a)$, it follows that
\[
  H_{m,n}
  \simeq
  H_{m,n,\tors}
  \oplus
  H_{m,n,\tf},
\]
where
\[
  H_{m,n,\tf}:=H_{m,n}/H_{m,n,\tors}
\]
is a finite free $\mS$-module (see \cite[Theorem 5.11]{M21}). Thus the natural map
\[
  H_{m,n}\to H_{m,n,\tf}
\]
is split.

Now apply Lemma~\ref{glsplit} to the short exact
sequence \eqref{filshort}. The torsion row is split by
\eqref{torsrow}, and the right-hand term $H_{m,n}$ has a split
torsion/free decomposition with torsion-free quotient finite free over
$\mS$. Hence the surjection
\[
  \fM_{m,n}\to H_{m,n}
\]
is split. Therefore \eqref{filshort} is split exact for all
$m,n$.
\end{proof}
\begin{thm}\label{hpdege}
Let $X$ be a smooth proper variety over $\sO_K$, and put $d:=\dim(X/\sO_K)$. Assume $2de<p-1$. Then the spectral sequence
\[
  E_2^{i,j}
  =
  H_{\dR}^{i-j}(X/\sO_K)
  \Longrightarrow
  \pi_{-i-j}\bigl(
    \HP(X/\sO_K;\Z_p)
  \bigr)
\]
is split degenerate. 
\end{thm}
\begin{proof}
Put $\sT:=\perf(X)$. The map $\bS[z]\to \sO_K$ sending $z$ to $\pi$ induces a morphism
\begin{equation}\label{tcnhpcmp}
\TCn(\sT/\bS[z];\Z_p)\otimes^{\L}_{\TP(\sO_K/\bS[z];\Z_p)} \HP(\sO_K/\sO_K;\Z_p)
\longrightarrow
\HP(\sT/\sO_K;\Z_p).
\end{equation}
By \cite[Proposition 2.16]{M26}, the left-hand side of \eqref{tcnhpcmp} defines a symmetric monoidal functor from the category of smooth proper $\sO_K$-linear dg categories to $\Mod_{\HP(\sO_K/\sO_K;\Z_p)}(\Sp)$. On the other hand, for any idempotent-complete small smooth proper $\sO_K$-linear dg category $\sD$, the object $\HH(\sD/\sO_K;\Z_p)$ is dualizable in $\Mod_{H\sO_K}(\Sp^{BS^1})$, and by \cite[Theorem 2.15]{AMN18}, it is in fact perfect. It follows that the right-hand side of \eqref{tcnhpcmp} also defines a symmetric monoidal functor from the category of smooth proper $\sO_K$-linear dg categories to $\Mod_{\HP(\sO_K/\sO_K;\Z_p)}(\Sp)$; see \cite{AMN18}. Thus, \eqref{tcnhpcmp} defines a symmetric monoidal natural transformation between these two functors. Hence \eqref{tcnhpcmp} is an equivalence by \cite[Proposition 4.6]{AMN18}. 

By Theorem~\ref{bksplit}, the filtered object
$\TCn(\sX/\bS[z];\Z_p)[1/u]$ is split over $\mS$. In particular, the relevant homotopy groups have no $E$-torsion. Thus, on homotopy groups of the equivalence \eqref{tcnhpcmp}, we obtain
\begin{equation}\label{hpbase}
\TCn_n(\sX/\bS[z];\Z_p)[1/u] \otimes_{\mS}\sO_K \simeq \HP_n(X/\sO_K;\Z_p).
\end{equation}
Similarly, using \cite[Theorem 5.11]{M21}, we know $H_{\mS}^*(\sX)$ has no $E$-torsion. Thus, using \cite[Theorem 1.2]{BMS19}, we know
\[
  H_{\mS}^m(\sX) \otimes_{\mS}\sO_K \simeq H_{\dR}^m(X/\sO_K).
  \]
  Using Theorem~\ref{bksplit}, we obtain 
  \[
    \HP_n(X/\sO_K;\Z_p) \simeq \bigoplus_{i\in\Z}H_{\dR}^{-n+2i}(X/\sO_K).
    \]
    and by Proposition~\ref{sscrit}, we obtain the claim.
\end{proof}

\begin{theorem}\label{ainftpsplit}
 Let $X$ be a smooth proper variety over $\sO_K$, and put $d:=\dim(X/\sO_K)$. Assume $2de<p-1$. Then the $\Ainf$-to-$\TP$ spectral sequence is split degenerate. In
  particular, there is an abstract isomorphism of $\Ainf$-modules
  \[
    \TP_j(X;\Z_p)
    \simeq
    \bigoplus_{i\in \Z}
    H_{\Ainf}^{-j+2i}(X).
  \]
\end{theorem}
\begin{proof}
By Theorem~\ref{bksplit}, the Breuil--Kisin spectral sequence
\begin{equation}\label{bkssainf}
  E_2^{i,j}
  =
  H_{\mS}^{i-j}(\sX)\{j\}
  \Longrightarrow
  \pi_{-i-j}\bigl(
    \TCn(\sX/\bS[z];\Z_p)[1/u]
  \bigr)
\end{equation}
is split degenerate. Equivalently, for every $j,n$ the sequence
\[
  0
  \to
  \pi_j\fil_{\BMS}^{n+1}
  \TCn(\sX/\bS[z];\Z_p)[1/u]
  \to
  \pi_j\fil_{\BMS}^{n}
  \TCn(\sX/\bS[z];\Z_p)[1/u]
  \to
  H_{\mS}^{2n-j}(\sX)\{n\}
  \to
  0
\]
is split exact. Hence, since the filtration is finite in each homotopy
degree,
\begin{equation}\label{bksumainf}
  \pi_j\fil_{\BMS}^{n}
  \TCn(\sX/\bS[z];\Z_p)[1/u]
  \simeq
  \bigoplus_{i\geq n}
  H_{\mS}^{2i-j}(\sX)\{i\}.
\end{equation}

We now compare this filtered object with the $\Ainf$-filtered object. We use
the natural morphisms of filtered spectra
\[
  \fil_{\BMS}^{*}
  \TCn(\sX/\bS[z];\Z_p)[1/u]
  \longrightarrow
  \fil_{\BMS}^{*}
  \TP(\sX/\bS[z];\Z_p)
  \longrightarrow
  \fil_{\BMS}^{*}
  \TP(\sX_{\sO_{\sC}};\Z_p).
\]
On homotopy groups this map is $\varphi$-semilinear with respect to the
$\mS$-module structure on
\[
  \pi_j\fil_{\BMS}^{n}
  \TCn(\sX/\bS[z];\Z_p)[1/u].
\]
Therefore its $\Ainf$-linearization is the map
\begin{equation}\label{ainflin}
  \pi_j\fil_{\BMS}^{n}
  \TCn(\sX/\bS[z];\Z_p)[1/u]
  \otimes_{\mS,\ol{\phi}}\Ainf
  \longrightarrow
  \pi_j\fil_{\BMS}^{n}
  \TP(\sX_{\sO_{\sC}};\Z_p).
\end{equation}
By the comparison theorem between Breuil--Kisin cohomology and $\Ainf$
cohomology \cite[Theorem 1.2]{BMS19}, together with the compatibility of the above morphism with
graded pieces, \eqref{ainflin} is an isomorphism for
all $j,n$. Tensoring the split filtration \eqref{bksumainf} over
$\mS$ with $\Ainf$ and using the comparison
\eqref{ainflin}, we get the claim.
\end{proof}


\begin{lemma}\label{qdetect}
Let
\[
  R=\Z[S^{-1}],\qquad
  \Lambda=R[[q-1]].
\]
For a prime $\ell\notin S$, put
\[
  \Lambda_\ell:=\Lambda\widehat\otimes_R\Z_\ell
  \simeq
  \Z_\ell[[q-1]].
\]
Let $f:M\to N$ be a morphism of finitely generated $\Lambda$-modules.
Assume that
\[
  f\otimes_\Lambda \Lambda_\ell=0
\]
for every prime $\ell\notin S$. Then $f=0$.
\end{lemma}

\begin{proof}
Let $I:=\operatorname{im}(f)$. It is enough to show that $I=0$.
If $I\neq 0$, then since $\Lambda$ is noetherian, $I$ has a maximal ideal in
its support. Every maximal ideal of
\[
  \Lambda=R[[q-1]]
\]
is of the form
\[
  \mf{m}_\ell=(\ell,q-1)
\]
for some prime $\ell\notin S$. Hence
\[
  I_{\mf{m}_\ell}\neq 0
\]
for some $\ell\notin S$.

Let us denote by $\Lambda_{\mf{m}_\ell}$ the localization of $\Lambda$ at $\mf{m}_\ell$. Then the natural map
\[
  \Lambda_{\mf{m}_\ell}\longrightarrow \Lambda_\ell
\]
is faithfully flat; indeed $\Lambda_\ell\simeq \Z_\ell[[q-1]]$ is the
$\mf{m}_\ell$-adic completion of the noetherian local ring
$\Lambda_{\mf{m}_\ell}$. Therefore
\[
  I_{\mf{m}_\ell}\otimes_{\Lambda_{\mf{m}_\ell}}\Lambda_\ell
  \neq 0.
\]
But this is a subquotient of
\[
  I\otimes_\Lambda\Lambda_\ell
  =
  \operatorname{im}(f\otimes_\Lambda\Lambda_\ell),
\]
which is zero by assumption. This contradiction proves $I=0$.
\end{proof}

\begin{lemma}\label{qsplitlocal}
Let
\[
  R=\Z[S^{-1}],\qquad
  \Lambda=R[[q-1]],
\]
and for each prime $\ell\notin S$ put
\[
  \Lambda_\ell:=\Lambda\widehat\otimes_R\Z_\ell
  \simeq
  \Z_\ell[[q-1]].
\]
Let
\[
  0\to A\to B\to C\to 0
\]
be a short exact sequence of finitely generated $\Lambda$-modules. Assume
that for every prime $\ell\notin S$, the base changed sequence
\[
  0
  \to
  A\otimes_\Lambda\Lambda_\ell
  \to
  B\otimes_\Lambda\Lambda_\ell
  \to
  C\otimes_\Lambda\Lambda_\ell
  \to
  0
\]
is split exact. Then the original sequence is split exact.
\end{lemma}

\begin{proof}
Let
\[
  \xi\in \operatorname{Ext}^1_\Lambda(C,A)
\]
be the extension class. Since $\Lambda$ is noetherian and $A,C$ are
finitely generated, $\operatorname{Ext}^1_\Lambda(C,A)$ is a finitely
generated $\Lambda$-module.

We show that $\xi=0$. It is enough to show that $\xi$ vanishes after
localizing at every maximal ideal of $\Lambda$. Every maximal ideal of
$\Lambda=R[[q-1]]$ is of the form
\[
  \mf{m}_\ell=(\ell,q-1)
\]
for some $\ell\notin S$.

Fix such an $\ell$. By assumption, the sequence becomes split after base
change to
\[
  \Lambda_\ell\simeq \Z_\ell[[q-1]].
\]
Equivalently, the image of
\[
  \xi_{\mf{m}_\ell}
  \in
  \operatorname{Ext}^1_{\Lambda_{\mf{m}_\ell}}
  (C_{\mf{m}_\ell},A_{\mf{m}_\ell})
\]
vanishes after tensoring with $\Lambda_\ell$. Since
\[
  \Lambda_{\mf{m}_\ell}\to \Lambda_\ell
\]
is faithfully flat, this implies
\[
  \xi_{\mf{m}_\ell}=0.
\]
Thus $\xi$ vanishes after localization at every maximal ideal of $\Lambda$.
Therefore $\xi=0$, and the short exact sequence splits.
\end{proof}

The factorial in the following theorem is used only to ensure that every prime
$\ell$ not inverted in the coefficient ring satisfies both $\ell\nmid N$ and
$\ell>2d+1$.  After $\ell$-adic completion this puts us in the unramified
Breuil--Kisin situation over $\Z_\ell$ and gives the low-dimension
inequality $2d<\ell-1$ needed to apply Theorem~\ref{ainftpsplit}.

\begin{theorem}\label{qtpsplit}
Let $N$ be a positive even integer. Let $X$ be a smooth proper variety over $\Z[1/N]$, and put $d:=\dim(X/\Z)$. Assume that there exists a derived scheme $\mathcal X$ over $\bS[1/N]$ together with an equivalence
\[
  \mathcal X\otimes_{\bS}\Z\simeq X.
\]
Then, after inverting $(2d+1)!$, the $S^1$-equivariant even filtration on $\TP(\mathcal X\otimes\ku/\ku)$ is split degenerate. In particular,
\[
  \pi_j\TP(\mathcal X\otimes\ku/\ku)\Bigl[\frac{1}{(2d+1)!}\Bigr]
  \simeq
  \bigoplus_{i\in\Z}H_{\qdR}^{-j+2i}(X/\Z)\Bigl[\frac{1}{(2d+1)!}\Bigr].
\]
\end{theorem}

\begin{proof}
Put
\[
  d:=\dim(X/\Z)
\]
and
\[
  R:=\Z\Bigl[\frac{1}{N},\frac{1}{(2d+1)!}\Bigr],
  \qquad
  \Lambda:=R[[q-1]].
\]
Let
\[
  \fil_{\ev}^{*}
  \TP(\mathcal X\otimes\ku/\ku)
\]
denote the $S^1$-equivariant even filtration; see \cite{HRW22} and \cite{P25}. For fixed $j,n$, set
\[
  M_{j,n}:=
  \pi_j\fil_{\ev}^{n}
  \TP(\mathcal X\otimes\ku/\ku)
  \otimes_{\Z[1/N][[q-1]]}\Lambda
\]
and
\[
  G_{j,n}:=
  H_{\qdR}^{2n-j}(X/\Z)
  \otimes_{\Z[1/N][[q-1]]}\Lambda .
\]
The associated graded pieces of the even filtration are given by
$q$-de Rham cohomology. Hence there are exact triangles whose homotopy
long exact sequences give connecting maps
\[
  \partial_{j,n}:G_{j,n}\longrightarrow M_{j-1,n+1}.
\]
To prove split degeneration over $\Lambda$, it suffices to show that all
$\partial_{j,n}$ vanish and that the resulting short exact sequences
\begin{equation}\label{qshort}
  0
  \to
  M_{j,n+1}
  \to
  M_{j,n}
  \to
  G_{j,n}
  \to
  0
\end{equation}
are split exact.

Let $\ell$ be a prime not inverted in $R$. Then
\[
  \ell\nmid N
  \qquad\text{and}\qquad
  \ell>2d+1.
\]
Put
\[
  \Lambda_\ell
  :=
  \Lambda\widehat\otimes_R\Z_\ell .
\]
Since
\[
  \Z[[q-1]]\widehat\otimes_{\Z}\Z_\ell
  \simeq
  \Z_\ell[[q-1]],
\]
we have
\[
  \Lambda_\ell\simeq \Z_\ell[[q-1]].
\]
Under the identification
\[
  q-1=u,
\]
this is the Breuil--Kisin prism
\[
  \mS_\ell=\Z_\ell[[u]]
\]
for the unramified base $\Z_\ell$.

We write
\[
  \Ainf^{(\ell)}:=W(\sO_{\mathbb{C}_\ell}^{\flat}),
  \]
and one can regard $\Ainf^{(\ell)}$ as a $\mS_\ell$-algebra via the identity on $\Z_\ell$ and sending $z$ to $[\pi^{\flat}]^p$.

By the functoriality of $S^1$-equivariant even filtration, a $S^1$-equivariant morphism 
\[
    \THH(\mathcal X\otimes\ku/\ku) \to \THH(X_{\sO_{\mathbb{C}_\ell}};\Z_\ell)
  \]
  induces a morphism of filtered spectra
\begin{equation}\label{qbkcomp}
  \fil_{\ev}^{*}
  \TP(\mathcal X\otimes\ku/\ku)
  \to
  \fil_{\ev}^{*}
  \TP(X_{\sO_{\mathbb{C}_\ell}};\Z_\ell).
\end{equation}
Here, by \cite[Theorem 5.0.3]{HRW22} the right-hand side is equivalent to 
\[
    \fil_{\BMS}^{*}
  \TP(X_{\sO_{\mathbb{C}_\ell}};\Z_\ell).
  \]
On associated graded pieces, \eqref{qbkcomp} identifies
$q$-de Rham cohomology with $\Ainf^{(\ell)}$-cohomology (see \cite[Theorem 17.2]{BS22}):
\[
  H_{\qdR}^{q}(X/\Z)
  \otimes_{\Z[1/N][[q-1]]}
  \Ainf^{(\ell)}
  \simeq
  H_{\Ainf^{(\ell)}}^{q}(X_{\sO_{\mathbb{C}_\ell}}).
\]
This comparison includes the isomorphism of $\Ainf^{(\ell)}$-modules for any $m,n\in\Z$
\begin{equation}\label{qbkcompisom}
  \pi_m\fil_{\ev}^{n}
  \TP(\mathcal X\otimes\ku/\ku) \otimes_{\Z[1/N][[q-1]]}
  \Ainf^{(\ell)}
  \simeq
  \pi_m\fil_{\ev}^{n}
  \TP(X_{\sO_{\mathbb{C}_\ell}};\Z_\ell).
\end{equation}

Now apply Theorem~\ref{ainftpsplit} to
$X_{\Z_\ell}$ over $\Z_\ell$. Since the ramification index is $e=1$ and
\[
  2d<\ell-1,
\]
the low ramification condition is satisfied. Therefore the filtration $\fil_{\ev}^{*}\TP(X_{\sO_{\mathbb{C}_\ell}};\Z_\ell)\simeq \fil_{\BMS}^{*}\TP(X_{\sO_{\mathbb{C}_\ell}};\Z_\ell)$ is split degenerate.

Via the filtered comparison \eqref{qbkcompisom}, since $\mS_\ell \to \Ainf^{(\ell)}$ is faithfully flat, it follows
that the $\Lambda_\ell$-base change of the global even filtration is split
degenerate. In particular,
\[
  \partial_{j,n}\otimes_{\Lambda}\Lambda_\ell=0
\]
for every prime $\ell$ not inverted in $R$.

By Lemma~\ref{qdetect}, we get
\[
  \partial_{j,n}=0
\]
for all $j,n$. Thus \eqref{qshort} is a short
exact sequence.

It remains to prove that this short exact sequence splits. But after base
change to $\Lambda_\ell$ it is split exact for every prime $\ell$ not inverted
in $R$. Therefore
Lemma~\ref{qsplitlocal} implies that
\[
  0
  \to
  M_{j,n+1}
  \to
  M_{j,n}
  \to
  G_{j,n}
  \to
  0
\]
is split exact over $\Lambda$.
\end{proof}


\bibliography{bib}
\bibliographystyle{unsrt}


\end{document}